\documentclass[twoside,11pt]{article}
\usepackage{amssymb}
\usepackage{bm}
\usepackage{bbm}
\usepackage{mathtools}
\usepackage{enumitem}

\usepackage{pgfplots}
\usepackage{amsmath}
\usepackage{booktabs}
\usepackage{amssymb}
\usepackage[hidelinks]{hyperref}
\usepackage{bm}
\newcommand{\ceil}[1]{\left\lceil #1 \right\rceil}
\usepackage{bbm}
\usepackage{mathtools}
\usepackage{algorithm}
\usepackage{algorithmic}
\usepackage{booktabs}       
\usepackage{amsfonts}
\usepackage[mathscr]{euscript}
\usepackage{subcaption}

\usepackage{xcolor}

\usepackage{jmlr2e}

\jmlrheading{1}{2020}{1-48}{4/00}{10/00}{Dimitris Bertsimas, Ryan Cory-Wright and Jean Pauphilet}

\ShortHeadings{Solving Large-Scale Sparse PCA to Certifiable (Near) Optimality}{Bertsimas, Cory-Wright and Pauphilet}
\firstpageno{1}

\begin{document}

\title{Solving Large-Scale Sparse PCA to Certifiable (Near) Optimality}

\author{\name Dimitris Bertsimas \email dbertsim@mit.edu \\
       \addr Massachusetts Institute of Technology\\
       Cambridge, MA 02139, USA
       \AND
       \name Ryan Cory-Wright \email ryancw@mit.edu \\
       \addr Massachusetts Institute of Technology\\
       Cambridge, MA 02139, USA
       \AND
       \name Jean Pauphilet \email jpauphilet@london.edu \\
       \addr London Business School\\
       London, UK}

\editor{TBD}

\maketitle

\begin{abstract}
Sparse principal component analysis (PCA) is a popular dimensionality reduction technique for obtaining principal components which are linear combinations of a small subset of the original features.  Existing 
approaches  cannot supply certifiably optimal principal components with more than $p=100s$ {of variables}.
By reformulating sparse PCA as a convex mixed-integer semidefinite optimization problem, we design a cutting-plane method which solves the problem to certifiable optimality at the scale of selecting {\color{black}$k=5$} covariates from $p=300$ variables, and provides small bound gaps at a larger scale. We also propose {\color{black}a} convex relaxation and {\color{black}greedy} rounding scheme that provides {\color{black} bound gaps of $1-2\%$ in practice} within minutes for $p=100$s or hours for $p=1,000$s {\color{black}and is therefore a viable alternative to the exact method at scale}.
Using real-world financial and medical datasets, we illustrate our approach's ability to derive interpretable principal components tractably at scale.
\end{abstract}

\begin{keywords}
  Sparse PCA, Sparse Eigenvalues, Semidefinite Optimization
\end{keywords}

\section{Introduction}
In the era of big data, interpretable methods for compressing a high-dimensional dataset into a lower dimensional set which shares the same essential characteristics are imperative. Since the work of \citet{hotelling1933analysis}, principal component analysis (PCA) has been one of the most popular approaches for completing this task. Formally, given centered data $\bm{A} \in \mathbb{R}^{n \times p}$ and its normalized empirical covariance matrix $\bm{\Sigma}:=\frac{\bm{A} \bm{A}^\top }{{n-1}} \in \mathbb{R}^{p \times p}$, PCA selects one or more leading eigenvectors of $\bm{\Sigma}$ and subsequently projects $\bm{A}$ onto these eigenvectors. This can be achieved in $O(p^3)$ time by taking a singular value decomposition {\color{black}$\bm{\Sigma}=\bm{U}\bm{\Lambda}\bm{U}^\top$}.

A common criticism\footnote{\color{black}A second criticism of PCA is that, as set up here, it uses the sample correlation or covariance matrix. This is a drawback, because sample covariance matrices are poorly conditioned estimators which over-disperses the sample eigenvalues, particularly in high-dimensional settings. In practice, this can be rectified by, e.g., using a shrinkage estimator \citep[see, e.g.,][]{ledoit2004well}. We do not do so here for simplicity, but we recommend doing so if using the techniques developed in this paper in practice.} of PCA is that the columns of {\color{black}$\bm{U}$} are not interpretable, since each eigenvector is a linear combination of all $p$ original features. This causes difficulties because:
\begin{itemize}\itemsep0em
    \item {\color{black}In medical applications such as cancer detection, PCs generated during exploratory data analysis need to supply interpretable modes of variation \citep{hsu2014sparse}.}
    \item In scientific applications such as protein folding, each original co-ordinate axis has a physical interpretation, and the reduced set of co-ordinate axes should too.
    \item In finance applications such as investing capital across index funds, each non-zero entry in each eigenvector used to reduce the feature space incurs a transaction cost.
    \item If $p \gg n$, PCA suffers from a curse of dimensionality and becomes physically meaningless \citep{amini2008high}.
\end{itemize}

One common method for obtaining interpretable principal components is to stipulate that they are sparse, i.e., maximize variance while containing at most $k$ non-zero entries. This approach leads to the following non-convex mixed-integer quadratically constrained problem \citep[see][]{d2005direct}:
\begin{align}\label{OriginalSPCA}
    \max_{\bm{x} \in \mathbb{R}^p} \ \bm{x}^\top \bm{\Sigma} \bm{x} \quad
    \text{s.t.} \quad \bm{x}^\top \bm{x}= 1,\ \vert \vert \bm{x} \vert \vert_0 \leq k,
\end{align}
where the constraint $\vert \vert \bm{x} \vert \vert_0 \leq k$ forces variance to be explained in a compelling fashion.

\subsection{Background and Literature Review}
Owing to sparse PCA's fundamental importance in a variety of applications including best subset selection \citep{d2008optimal}, natural language processing \citep{zhang2012sparse}, compressed sensing \citep{candes2007dantzig}, and clustering \citep{luss2010clustering}, three distinct classes of methods for addressing Problem \eqref{OriginalSPCA} have arisen. Namely, (a) heuristic methods which obtain high-quality sparse PCs in an efficient fashion but do not supply guarantees on the quality of the solution, (b) convex relaxations which obtain certifiably near-optimal solutions by solving a convex relaxation and rounding, and (c) exact methods which obtain certifiably optimal solutions, albeit in exponential time.

\paragraph{Heuristic Approaches: }
The importance of identifying a small number of interpretable principal components has been well-documented in the literature since the work of \citet{hotelling1933analysis} \citep[see also][]{jeffers1967two}, giving rise to many distinct heuristic approaches for obtaining high-quality solutions to Problem \eqref{OriginalSPCA}. Two interesting such approaches are to rotate dense principal components to promote sparsity \citep{kaiser1958varimax, richman1986rotation, jolliffe1995rotation}, or apply an $\ell_1$ penalty term as a convex surrogate to the cardinality constraint \citep{jolliffe2003modified, zou2006sparse}. Unfortunately, the former approach does not provide performance guarantees, while the latter approach {still results in} a non-convex optimization problem.

More recently, motivated by the need to rapidly obtain high-quality sparse principal components at scale, a wide variety of first-order heuristic methods have emerged. The first such \textit{modern} heuristic was developed by \citet{journee2010generalized}, and involves combining the power method with thresholding and re-normalization steps. By pursuing similar ideas, several related methods have since been developed \citep[see][]{witten2009penalized, hein2010inverse, richtarik2012alternating, luss2013conditional, yuan2013truncated}. Unfortunately, while these methods are often very effective in practice, they sometimes badly fail to recover an optimal sparse principal component, and a practitioner using a heuristic method typically has no way of knowing when this has occurred. Indeed, \citet{berk2017} recently compared $7$ heuristic methods, including most of those reviewed here, on $14$ instances of sparse PCA, and found that none of the heuristic methods successfully recovered an optimal solution in all $14$ cases {(i.e., no heuristic was right all the time).}

\paragraph{Convex Relaxations: }
Motivated by the shortcomings of heuristic approaches on high-dimensional datasets, and the successful application of semi-definite optimization in obtaining high-quality approximation bounds in other applications \citep[see][]{goemans1995improved, wolkowicz2012handbook}, a variety of convex relaxations have been proposed for sparse PCA. The first such convex relaxation was proposed by  \citet{d2005direct}, who reformulated sparse PCA as the rank-constrained mixed-integer semidefinite optimization problem (MISDO):
\begin{equation}\label{sdospca1}
\begin{aligned}
    \max_{\bm{X} \succeq \bm{0}} \quad \langle \bm{\Sigma}, \bm{X} \rangle\ \text{s.t.} \ \mathrm{tr}(\bm{X})=1, \ \Vert\bm{X}\Vert_0 \leq k^2,\ \mathrm{Rank}(\bm{X})=1,
\end{aligned}
\end{equation}
where $\bm{X}$ models the outer product $\bm{x}\bm{x}^\top$. {\color{black}Note that, for a rank-one matrix $\bm{X}$, the constraint $\Vert \bm{X}\Vert_0 \leq k^2$ in \eqref{sdospca1} is equivalent to the constraint $\Vert\bm{x}\Vert_0 \leq k$ in \eqref{OriginalSPCA}, since a vector $\bm{x}$ is $k$-sparse if its outer product $\bm{x}\bm{x}^\top$ is $k^2$-sparse.}
{After performing this reformulation,} \citet{d2005direct} relaxed both the cardinality and rank constraints and instead solved
\begin{equation}\label{sdospca1.relax}
\begin{aligned}
    \max_{\bm{X} \succeq \bm{0}} \quad \langle \bm{\Sigma}, \bm{X} \rangle\ \text{s.t.} \ \mathrm{tr}(\bm{X})=1, \ \Vert\bm{X}\Vert_1 \leq k,
\end{aligned}
\end{equation}
which supplies a valid upper bound on Problem \eqref{OriginalSPCA}'s objective.

The semidefinite approach has since been refined in a number of follow-up works. Among others, \citet{d2008optimal}, building upon the work of \citet{ben2002tractable}, proposed a different semidefinite relaxation which supplies a sufficient condition for optimality via the primal-dual KKT conditions, and \citet{d2014approximation} analyzed the quality of the semidefinite relaxation in order to obtain high-quality approximation bounds. A common theme in these approaches is that they require solving large-scale semidefinite optimization problems. This presents difficulties for practitioners because state-of-the-art implementations of interior point methods such as \verb|Mosek| require $O(p^6)$ memory to solve Problem \eqref{sdospca1.relax}, and therefore currently cannot solve instances of Problem \eqref{sdospca1.relax} with $p \geq 300$ \citep[see][for a recent comparison]{bertsimas2019polyhedral}. {\color{black}Techniques other than interior point methods, e.g., ADMM or augmented Lagrangian methods as reviewed in \cite{majumdar2019survey} could also be used to solve Problem \eqref{sdospca1.relax}, although they tend to require more runtime than IPMs to obtain a solution of a similar accuracy and be numerically unstable for problem sizes where IPMs run out of memory \citep{majumdar2019survey}.}

{A number of works have also studied the statistical estimation properties of Problem \eqref{sdospca1.relax}, by assuming an underlying probabilistic model. Among others, \citet{amini2008high} have demonstrated the asymptotic consistency of Problem \eqref{sdospca1.relax} under a spiked covariance model once the number of samples used to generate the covariance matrix exceeds a certain threshold; see \cite{vu2012minimax, berthet2013optimal, wang2016statistical} for further results in this direction, \cite{miolane2018phase} for a recent survey.}

{ In an complementary direction, \citet{dey2018convex} has recently questioned the modeling paradigm of lifting $\bm{x}$ to a higher dimensional space by instead considering the following (tighter) relaxation of sparse PCA in the original problem space
\begin{align}\label{prob:l1relax_small}
    \max_{\bm{x} \in \mathbb{R}^p}\quad \bm{x}^\top \bm{\Sigma}\bm{x}\quad \text{s.t.}\quad \Vert\bm{x}\Vert_2=1, {\color{black}\Vert \bm{x}\Vert_1 \leq \sqrt{k}}.
\end{align}

Interestingly, Problem \eqref{prob:l1relax_small}'s relaxation provides a $\left(1+\sqrt{\tfrac{k}{k+1}}\right)^2$-factor bound approximation of Problem \eqref{OriginalSPCA}'s objective, while Problem \eqref{sdospca1.relax}'s upper bound may be exponentially larger in the worst case \citep{amini2008high}. This additional tightness, however, comes at a price: Problem \eqref{prob:l1relax_small} is NP-hard to solve—indeed, providing a constant-factor guarantee on sparse PCA is NP-hard \citep{magdon2017np}—and thus \eqref{prob:l1relax_small} is best formulated as a MIO, while Problem \eqref{sdospca1.relax} can be solved in polynomial time.
}

More recently, by building on the work of \citet{kim2001second}, \citet[]{bertsimas2019polyhedral} introduced a second-order cone relaxation of \eqref{sdospca1} which scales to $p=1000s$, and matches the semidefinite bound after imposing a small number of cuts. Moreover, it typically supplies bound gaps of less than $5\%$. However, it does not supply an \textit{exact} certificate of optimality, which is often desirable, for instance in medical applications. 

A fundamental drawback of existing convex relaxation techniques is that they are not coupled with rounding schemes for obtaining high-quality feasible solutions. This is problematic, because optimizers are typically interested in obtaining high-quality solutions, rather than certificates. In this paper, we take a step in this direction, by deriving new convex relaxations that naturally give rise to greedy and random rounding schemes. The fundamental point of difference between our relaxations and existing relaxations is that we derive our relaxations by rewriting sparse PCA as a MISDO and dropping an integrality constraint, rather than using more ad-hoc techniques.

\paragraph{Exact Methods:} Motivated by the successful application of mixed-integer optimization for solving statistical learning problems such as best subset selection \citep{bertsimas2020sparse} and sparse classification \citep{bertsimas2017sparse}, several exact methods for solving sparse PCA to certifiable optimality have been proposed. The first branch-and-bound algorithm for solving Problem \eqref{OriginalSPCA} was proposed by \citet{moghaddam2006spectral}, by applying norm equivalence relations to obtain valid bounds. However, \citet{moghaddam2006spectral} did not couple their approach with high-quality initial solutions and tractable bounds to prune partial solutions. Consequently, they could not scale their approach beyond $p=40$.

A more sophisticated branch-and-bound scheme was recently proposed by \citet{berk2017}, which couples tighter Gershgorin Circle Theorem bounds \citep[][Chapter 6]{horn1990matrix} with a fast heuristic due to \cite{yuan2013truncated} to solve problems up to $p=250$. However, their method cannot scale beyond $p=100$s, because the bounds obtained are too weak to avoid enumerating a sizeable portion of the tree.

Recently, the authors developed a framework for reformulating convex mixed-integer optimization problems with logical constraints \citep[see][]{bertsimas2019unified}, and demonstrated that this framework allows a number of problems of practical relevance to be solved to certifiably optimality via a cutting-plane method. In this paper, we build upon this work by reformulating Problem \eqref{OriginalSPCA} as a \textit{convex} mixed-integer semidefinite optimization problem, and leverage this reformulation to design a cutting-plane method which solves sparse PCA to certifiable optimality. A key feature of our approach is that we need not solve any semidefinite subproblems. Rather, we use {concepts} from SDO to design a semidefinite-free approach which uses simple linear algebra techniques.

{ Concurrently to our initial submission, \citet{li2020exact} also attempted to reformulate sparse PCA as an MISDO, and proposed valid inequalities for strengthening their formulation and local search algorithms for obtaining high-quality solutions at scale. Our work differs in the following two ways. First, we propose strengthening the MISDO formulation using the Gershgorin circle theorem and demonstrate that this allows our MISDO formulation to scale to problems with $p=100$s of features, while they do not, to our knowledge, solve any MISDOs to certifiable optimality where $p>13$. Second, we develop tractable second-order cone relaxations and greedy rounding schemes which allow practitioners to obtain certifiably near optimal sparse principal components even in the presence of $p=1,000$s of features. More remarkable than the differences between the works however is the similarities: more than $15$ years after \citet{d2005direct}'s landmark paper first appeared, both works proposed reformulating sparse PCA as an MISDO less than a week apart. In our view, this demonstrates that the ideas contained in both works transcend sparse PCA, and can perhaps be applied to other problems in the optimization literature which have not yet been formulated as MISDOs.}

\subsection{Contributions and Structure}
The main contributions of the paper are twofold. First, we reformulate sparse PCA exactly as a mixed-integer semidefinite optimization problem; a reformulation which is, to the best of our knowledge, novel. Second, we leverage this MISDO formulation to design efficient algorithms for solving non-convex mixed-integer quadratic optimization problems, such as sparse PCA, to certifiable optimality or {\color{black}within $1-2\%$ of optimality in practice} at a larger scale than existing state-of-the-art methods.
The structure and detailed contributions of the paper are as follows:
\begin{itemize}\itemsep0em
    \item In Section \ref{sec:exact.misdo}, we reformulate Problem \eqref{OriginalSPCA} as a mixed-integer SDO.  { We propose a cutting-plane method which solves it to certifiable optimality in Section \ref{sec:exact.oa}. Our algorithm decomposes the problem into a purely binary master problem and a semidefinite separation problem. Interestingly, we show in Section \ref{sec:exact.subpb} that the separation problems can be solved efficiently via a leading eigenvalue computation and does not require any SDO solver. Finally, {\color{black}the} Gershgorin Circle theorem has been empirically successful for deriving upper-bounds on the objective value of \eqref{OriginalSPCA} \citep{berk2017}. We theoretically analyze the quality of such bounds in Section \ref{sec:exact.circle} and show in Section \ref{sec:exact.oval} that tighter bounds derived from Brauer's ovals of Cassini theorem can also be imposed via mixed-integer second-order cone constraints.}
    \item In Section \ref{sec:relaxandround}, we analyze the semidefinite reformulation's convex relaxation, and introduce a greedy rounding scheme (Section \ref{ssec:relax.bool}) which supplies high-quality solutions to Problem \eqref{OriginalSPCA} in polynomial time, {\color{black} together with a sub-optimality gap (see numerical experiments in Section \ref{sec:numres})}. To further improve the quality of rounded solution and the optimality gap, we introduce strengthening inequalities (Section \ref{ssec:validineq}). While solving the strengthened formulation exactly would result in an intractable MISDO problem, solving its relaxation and rounding the solution appears as an efficient strategy to return high-quality solutions with a {\color{black} numerical certificate} of near-optimality.
    \item In Section \ref{sec:numres}, we apply the cutting-plane and random rounding methods to derive optimal and near optimal sparse principal components for problems in the UCI dataset. We also compare our method's performance against the method of \citet{berk2017}, and find that our exact cutting-plane method performs comparably, while our relax+round approach successfully scales to problems an order of magnitude larger {\color{black}and often returns solutions which outperform the exact method at sizes which the exact method cannot currently scale to}. A key feature of our numerical success is that we sidestep the computational difficulties in solving SDOs at scale by proposing semidefinite-free methods for solving the convex relaxations, i.e., solving second-order cone relaxations.
\end{itemize}
\paragraph{Notation: }
We let nonbold face characters such as $b$ denote scalars, lowercase bold faced characters such as $\bm{x}$ denote vectors, uppercase bold faced characters such as $\bm{X}$ denote matrices, and calligraphic uppercase characters such as $\mathcal{Z}$ denote sets. We let $[p]$ denote the set of running indices $\{1, ..., p\}$. We let $\mathbf{e}$ denote a vector of all $1$'s, $\bm{0}$ denote a vector of all $0$'s, and $\mathbb{I}$ denote the identity matrix, with dimension implied by the context.

We also use an assortment of matrix operators. We let $\langle \cdot,\cdot \rangle$ denote the Euclidean inner product between two matrices, $\Vert \cdot \Vert_F$ denote the Frobenius norm of a matrix, $\Vert \cdot \Vert_\sigma$ denote the spectral norm of a matrix, $\Vert \cdot \Vert_*$ denote the nuclear norm of a matrix, $\bm{X}^\dag$ denote the Moore-Penrose psuedoinverse of a matrix $\bm{X}$ and $S_+^p$ denote the $p \times p$ positive semidefinite cone; see \citet{horn1990matrix} for a general theory of matrix operators. 

\section{An Exact Mixed-Integer Semidefinite {\color{black} Optimization Algorithm}}\label{sec:reformulation}
In Section \ref{sec:exact.misdo}, we reformulate Problem \eqref{OriginalSPCA} as a convex mixed-integer semidefinite optimization problem. From this formulation, we propose an outer-approximation scheme (Section \ref{sec:exact.oa}) which, as we show in Section \ref{sec:exact.subpb}, does not require solving any semidefinite problems. We improve convergence of the algorithm by deriving quality upper-bounds on Problem's \eqref{OriginalSPCA} objective value in Section \ref{sec:exact.circle} and \ref{sec:exact.oval}.
{\color{black}
\subsection{A Mixed-Integer Semidefinite Reformulation} \label{sec:exact.misdo}
}
Starting from the rank-constrained SDO formulation \eqref{sdospca1}, we introduce binary variables $z_i$ to model whether $X_{i,j}$ is non-zero, via the logical constraint $X_{i,j}=0$ if $z_i=0$; note that we need not require that $X_{i,j}=0$ if $z_j=0$, since $\bm{X}$ is a symmetric matrix. By enforcing the logical constraint via $-M_{i,j}z_i \leq X_{i,j}\leq M_{i,j}z_i$ for sufficiently large $M_{i,j}>0$, Problem \eqref{sdospca1} becomes
\begin{align*}
    \max_{\bm{z} \in \{0, 1\}^p: \bm{e}^\top \bm{z} \leq k} \ \max_{\bm{X} \in S^p_+} \quad & \langle \bm{\Sigma}, \bm{X} \rangle\\
    \text{s.t.} \quad & \mathrm{tr}(\bm{X})=1,\ -M_{i,j}z_i \leq X_{i,j} \leq M_{i,j}z_i \ \forall i, j \in [p],\ \mathrm{Rank}(\bm{X})=1.
\end{align*}

To obtain a MISDO reformulation, we omit the rank constraint. In general, omitting a rank constraint generates a relaxation and induces some loss of optimality. Remarkably, this omission 
is without loss of optimality in this case. Indeed, the objective is convex and therefore some rank-one extreme matrices $\bm{X}$ is optimal. We formalize this observation in the following theorem; note that a similar result—although in the context of computing Restricted Isometry constants and with a different proof—exists \citep[][]{gally2016computing}:

\begin{theorem}\label{thm:misdpreformthm}
Problem \eqref{OriginalSPCA} attains the same optimal objective value as the problem:
\begin{equation}\label{misdpprimal}
\begin{aligned}
    \max_{\bm{z} \in \{0, 1\}^p: \bm{e}^\top \bm{z} \leq k} \ \max_{\bm{X} \in S^p_+} \quad & \langle \bm{\Sigma}, \bm{X} \rangle&\\ \text{s.t.} \quad &  \mathrm{tr}(\bm{X})=1 \quad & \\
    & \vert X_{i,j}\vert \leq M_{i,j}z_i \quad &  \forall i, j \in [p],\\
    & \color{black}\sum_{j=1}^p \vert X_{i,j}\vert \leq \sqrt{k} z_i \quad &  \forall i \in [p],
\end{aligned}
\end{equation}
where $M_{i,i}=1$, {\color{black}and} $M_{i,j}=\frac{1}{2}$ if $j \neq i$. 
\end{theorem}
\begin{remark}
Observe that if {\color{black} $k \leq \sqrt{n}$ and} we set $M_{i,j}=1 \ \forall i,j \in [p]$ in Problem \eqref{misdpprimal} {\color{black}and omit the valid inequality $\sum_{j=1}^p \vert X_{i,j}\vert \leq \sqrt{k} z_i$} then the optimal value of the continuous relaxation is trivially $\lambda_{\max}(\bm{\Sigma})$. Indeed, letting $\bm{x}$ be a leading eigenvector of the unconstrained problem (where $\Vert \bm{x}\Vert_2=1$), we can set $z_i=\vert x_i\vert \geq \vert x_i \vert \vert x_j\vert${\color{black}, where the inequality holds since $\Vert \bm{x}\Vert_2=1$,} and $X_{i,j}=x_ix_j$, meaning {\color{black}(a) $\sum_i z_i=\Vert \bm{x}\Vert_1 \leq k \leq \sqrt{n}$ {\color{black}by norm equivalence} and (b) $\vert X_{i,j}\vert \leq z_i$} and thus $(\bm{X}, \bm{z})$ solves this continuous relaxation. Therefore, setting $M_{i,j}=\frac{1}{2}$ if $j\neq i$ {\color{black}and/or imposing the valid inequality $\sum_{j=1}^p \vert X_{i,j}\vert \leq \sqrt{k} z_i$} is necessary for obtaining non-trivial relaxations {\color{black}whenever $k$ is small}.
\end{remark}
{\color{black}
}


\begin{proof}
It suffices to demonstrate that for any feasible solution to \eqref{OriginalSPCA} we can construct a feasible solution to \eqref{misdpprimal} with an equal or greater payoff, and vice versa.
\begin{itemize}\itemsep0em
    \item Let $\bm{x} \in \mathbb{R}^{p}$ be a feasible solution to \eqref{OriginalSPCA}. Then, {\color{black}since $\Vert \bm{x}\Vert_1 \leq \sqrt{k}$}, $(\bm{X}:=\bm{x}\bm{x}^\top, \bm{z})$ is a feasible solution to \eqref{misdpprimal} with equal cost, where $z_i=1$ if $\vert x_i\vert>0$, $z_i=0$ otherwise.
    \item Let $(\bm{X}, \bm{z})$ be a feasible solution to Problem \eqref{misdpprimal}, and let $\bm{X}=\sum_{i=1}^p \sigma_i \bm{x}_i\bm{x}_i^\top$ be a Cholesky decomposition of $\bm{X}$, where $\bm{e}^\top \bm{\sigma}=1, \bm{\sigma} \geq \bm{0}$, {\color{black}and $\Vert \bm{x}_i\Vert_2=1 \ \forall i \in [p]$}. Observe that $\Vert\bm{x}_i\Vert_0 \leq k\ \forall i \in [p], $ since we can perform the Cholesky decomposition on the submatrix of $\bm{X}$ induced by $\bm{z}$, and ``pad'' out the remaining entries of each $\bm{x}_i$ with $0$s to obtain the decomposition of $\bm{X}$. Therefore, let us set $\hat{\bm{x}}:=\arg\max_{i}[\bm{x}_i^\top \bm{\Sigma}\bm{x}_i]$. Then, $\hat{\bm{x}}$ is a feasible solution to \eqref{OriginalSPCA} with an equal or greater payoff.
\end{itemize}
Finally, we let $M_{i,i}=1$, $M_{i,j}=\frac{1}{2}$ if $i \neq j$, as the $2 \times 2$ minors imply $X_{i,j}^2 \leq X_{i,i}X_{j,j}\leq \frac{1}{4}$ whenever $i \neq j$ \citep[c.f.][Lemma 1]{gally2016computing}.
\end{proof}
Theorem \ref{thm:misdpreformthm} reformulates Problem \eqref{OriginalSPCA} as a mixed-integer SDO. Therefore, we can solve Problem \eqref{misdpprimal} using general branch-and-cut techniques for semidefinite optimization problems \citep[see][]{gally2018framework, kobayashi2019branch}. However, this approach is not scalable, as it comprises solving a large number of semidefinite subproblems and the community does not know how to efficiently warm-start interior point methods (IPMs) for SDOs.

Alternatively, we propose a saddle-point reformulation of Problem \eqref{misdpprimal} which avoids the computational difficulty of solving a large number of SDOs by exploiting problem structure, as we will show in Section \ref{sec:exact.subpb}. The following result reformulates Problem \eqref{misdpprimal} as a max-min saddle-point problem amenable to outer-approximation: 

\begin{theorem}\label{thm:saddlepointtheorem}
Problem \eqref{misdpprimal} attains the same optimal value as the following problem:
\begin{align}\label{prob:saddlepointproblem}
\max_{\bm{z} \in \{0, 1\}^p: \ \bm{e}^\top \bm{z} \leq k} \quad f(\bm{z})\\ \label{eqn:separation}
    \quad \text{ where } \quad f(\bm{z}):= \min_{\lambda \in \mathbb{R}, \bm{\alpha} \in \mathbb{R}^{p \times p} {\color{black},\bm{\beta} \in \mathbb{R}^p}} \quad & \lambda  +\sum_{i=1}^p {z}_i {\color{black}\left(\sum_{j=1}^p M_{i,j}\max(0, \vert \alpha_{i,j}\vert-\beta_i)+{\color{black}\sqrt{k}\beta_i}\right)}\\
    \text{s.t.} \quad & \lambda\mathbb{I}+\bm{\alpha} \succeq \bm{\Sigma}.\nonumber
\end{align}

\end{theorem}
\begin{remark}
The above theorem demonstrates that $f(\bm{z})$ is concave in $\bm{z}$, by rewriting it as the infimum of functions which are linear in $\bm{z}$ \citep[][]{boyd2004convex}.
\end{remark}

\begin{proof}
Let us {\color{black}introduce auxiliary variables $U_{i,j}$ to model the absolute value of $X_{i,j}$ and rewrite the inner optimization problem of \eqref{misdpprimal} as
\begin{equation}\label{prob:sparsepcainnerprimal}
\begin{aligned}\color{black}
    f(\bm{z}):= \quad & \max_{\bm{X} \succeq \bm{0}, \bm{U}} \quad & \langle \bm{\Sigma}, \bm{X}\rangle\\
    \text{s.t.} \quad &  \mathrm{tr}(\bm{X})=1, \quad & & [\lambda]\\\
    & \color{black} U_{i,j} \leq M_{i,j}z_i \ &\forall i, j \in [p], \quad & [\sigma_{i,j}]\\
    & \color{black} \vert X_{i,j}\vert \leq U_{i,j} \ &\forall i, j \in [p],\quad & [\alpha_{i,j}]\\
    & \color{black} \sum_{j=1}^p U_{i,j} \leq \sqrt{k} z_i &\forall i \in [p], \quad & [\beta_{i}]\\
\end{aligned}
\end{equation}
where we associate dual constraint multipliers with primal constraints in square brackets.}
For $\bm{z}$ such that $\bm{e}^\top \bm{z} \geq 1$, the maximization problem induced by $f(\bm{z})$ satisfies Slater's condition \citep[see, e.g.,][Chapter 5.2.3]{boyd2004convex}, strong duality applies and leads to
{\color{black}
\begin{align*}
    f(\bm{z}) = \min_{\substack{\lambda \\ \bm{\sigma},\bm{\alpha},\bm{\beta} \geq \bm{0}}} \quad & \lambda+\sum_{i,j}\sigma_{i,j}M_{i,j}z_i+\sum_{i=1}^p \beta_i \sqrt{k}z_i\\
    \text{s.t.} \quad & \lambda \mathbb{I}+\bm{\alpha} \succeq \bm{\Sigma}, \vert \alpha_{i,j}\vert \leq \sigma_{i,j}+\beta_i.
\end{align*}}
{\color{black}
We eliminate $\bm{\sigma}$ from the dual problem above by optimizing over $\sigma_{i,j}$ and setting $\sigma^\star_{i,j}=\max(0, \vert \alpha_{i,j}\vert-\beta_i)$.}

Note that for $\bm{z}=\bm{0}$, the primal subproblem is infeasible and the dual subproblem has objective $-\infty$, but this can safely be ignored since $\bm{z}=\bm{0}$ is certainly suboptimal. 
\end{proof}

\subsection{A Cutting-Plane Method} \label{sec:exact.oa}

Theorem \ref{thm:saddlepointtheorem} shows that evaluating $f(\bm{\hat{z}})$ yields the globally valid overestimator: $$f(\bm{z}) \leq f(\hat{\bm{z}})+\bm{g}_{\hat{\bm{z}}}^\top(\bm{z}-\hat{\bm{z}}),$$ where $\bm{g}_{\hat{\bm{z}}}$ is a supergradient of $f$ at $\bm{\hat{z}}$, at no additional cost. In particular, we have

$$g_{\hat{\bm{z}},i}={\color{black}\left(\sum_{j=1}^p M_{i,j}\max\left(0, \vert \alpha_{i,j}^\star(\bm{\hat{z}})\vert-\beta_i(\hat{\bm{z}})\right)+{\color{black}\sqrt{k}\beta_i(\hat{\bm{z}})}\right)},$$

where {\color{black}$\bm{\alpha}^\star(\hat{\bm{z}})$, $\bm{\beta}^\star(\hat{\bm{z}})$ constitutes an optimal choice of $(\bm{\alpha}, \bm{\beta})$ for a fixed $\bm{\hat{\bm{z}}}$}. This observation leads to an efficient strategy for maximizing $f(\bm{z})$: iteratively maximizing and refining a piecewise linear upper estimator of $f(\bm{z})$. This strategy is called outer-approximation (OA), and was originally proposed by \citet{duran1986outer}. OA works by iteratively constructing estimators of the following form at each {\color{black}iteration} $t$:
\begin{align}
    f^t(\bm{z})=\min_{1 \leq i \leq t} \left\{f(\bm{z}_i)+\bm{g}_{\bm{z}_i}^\top (\bm{z}-\bm{z}_i)\right\}.
\end{align}
After constructing each overestimator, we maximize $f^t(\bm{z})$ over $\{0, 1\}^p$ to obtain $\bm{z}_t$, and evaluate $f(\cdot)$ and its supergradient at $\bm{z}_t$. This procedure yields a non-increasing sequence of overestimators $\{f^t(\bm{z}_t)\}_{t=1}^T$ which converge to the optimal value of $f(\bm{z})$ within a finite number of iterations {\color{black}$T \leq {p \choose 1}+\ldots+{p \choose k}$}, since $\{0, 1\}^p$ is a finite set and OA never visits a point twice. Additionally, we can avoid solving a different MILO at each OA iteration by integrating the entire algorithm within a single branch-and-bound tree, as proposed by \cite{quesada1992lp}, using \verb|lazy constraint callbacks|. Lazy constraint callbacks are now standard components of modern MILO solvers such as \verb|Gurobi| or \verb|CPLEX| and substantially speed-up OA. We formalize this procedure in Algorithm \ref{alg:cuttingPlaneMethod}; note that $\partial f(\bm{z}_{t+1})$ denotes the set of supergradients of $f$ at $\bm{z}_{t+1}$.
\begin{algorithm*}
\caption{An outer-approximation method for Problem \eqref{OriginalSPCA}}
\label{alg:cuttingPlaneMethod}
\begin{algorithmic}\normalsize
\REQUIRE Initial solution $\bm{z}_1$
\STATE $t \leftarrow 1 $
\REPEAT
\STATE Compute $\bm{z}_{t+1}, \theta_{t+1}$ solution of
{\vspace{-2mm}
\begin{align*}
\max_{\bm{z} \in\{0, 1\}^p: \bm{e}^\top \bm{z} \leq k, \theta} \: \theta \quad \mbox{ s.t. } \theta \leq f(\bm{z}_i) + \bm{g}_{\bm{z}_i}^\top (\bm{z}-\bm{z}_i) \ \forall i \in [t],
\end{align*}}\vspace{-5mm}
\STATE Compute $f(\bm{z}_{t+1})$ and $\bm{g}_{\bm{z}_{t+1}} \in \partial f(\bm{z}_{t+1})$ by solving \eqref{eqn:separation}
\STATE $t \leftarrow t+1 $
\UNTIL{$ f(\bm{z}_t)-\theta_t \leq \varepsilon$}
\RETURN $\bm{z}_t$
\end{algorithmic}
\end{algorithm*}

\subsection{A { Semidefinite-free} Subproblem Strategy} \label{sec:exact.subpb}
Our derivation and analysis of Algorithm \ref{alg:cuttingPlaneMethod} indicates that we can solve Problem \eqref{OriginalSPCA} to certifiable optimality by solving a (potentially large) number of semidefinite subproblems \eqref{eqn:separation}, which might be prohibitive in practice. 
Therefore, we now derive a computationally efficient subproblem strategy which crucially does not require solving \textit{any} semidefinite programs. Formally, we have the following result:
\begin{theorem}\label{compefficientsubproblem}\color{black}
For any $\bm{z} \in \{0, 1\}^p: \bm{e}^\top \bm{z} \leq k$, optimal dual variables in \eqref{eqn:separation} are
\begin{align}\color{black}
    \lambda=\lambda_{\max}\left(\bm{\Sigma}_{1,1}\right),\
   \hat{\bm{\alpha}}=\begin{pmatrix} \hat{\bm{\alpha}}_{1,1} & \hat{\bm{\alpha}}_{1,2}\\ \hat{\bm{\alpha}}_{1,2}^\top & \hat{\bm{\alpha}}_{2,2}\end{pmatrix}=\begin{pmatrix} \bm{0} & \bm{0}\\ \bm{0} & \bm{\Sigma}_{2,2}-\lambda \mathbb{I}+\bm{\Sigma}_{1,2}^\top \left(\lambda \mathbb{I}-\bm{\Sigma}_{1,1}\right)^\dag \bm{\Sigma}_{1,2}\end{pmatrix},\\
   {\color{black}\beta_{i}=(1-z_i)\begin{pmatrix}\vert\hat{\alpha}_{i,1}\vert, \vert\hat{\alpha}_{i,2}\vert, \ldots, \vert\hat{\alpha}_{i,i}\vert, \vert\hat{\alpha}_{i,i}\vert, \ldots, \vert\hat{\alpha}_{i,p}\vert\end{pmatrix}_{[\ceil{\sqrt{k}\ }]} \ \forall i \in [p],}
\end{align}
where $\lambda_{\max}(\cdot)$ denotes the leading eigenvalue of a matrix,  $\hat{\bm{\alpha}}=\begin{pmatrix} \hat{\bm{\alpha}}_{1,1} & \hat{\bm{\alpha}}_{1,2}\\ \hat{\bm{\alpha}_{1,2}^\top} & \hat{\bm{\alpha}}_{2,2}\end{pmatrix}$ is a {\color{black}permutation} such that $\hat{\bm{\alpha}}_{1,1}$ (resp. $\hat{\bm{\alpha}}_{2,2}$) denotes the entries of $\hat{\bm{\alpha}}$ where $z_i=z_j=1$ ($z_i=z_j=0$); $\bm{\Sigma}$ is similar{\color{black}, and $(\bm{x})_{[k]}$ denotes the $k$th largest element of $\bm{x}$.}
\end{theorem}
\begin{remark}
By Theorem \ref{compefficientsubproblem}, Problem \eqref{eqn:separation} can be solved by computing the leading eigenvalue of $\bm{\Sigma}_{1,1}$ and solving a linear system. This justifies our claim that we need not solve any SDOs in our algorithmic strategy.
\end{remark}
\begin{proof}
We appeal to strong duality and complementary slackness. Observe that, for any $\bm{z} \in \{0, 1\}^p$, $f(\bm{z})$ is the optimal value of a maximization problem over a closed convex compact set. Therefore, there exists some optimal primal solution $\bm{X}^\star$ without loss of generality. Moreover, since the primal has non-empty relative interior {with respect to the non-affine constraints, it satisfies the Slater constraint qualification and} strong duality holds \citep[see, e.g.,][Chapter 5.2.3]{boyd2004convex}. Therefore, by complementary slackness \citep[see, e.g.,][Chapter 5.5.2]{boyd2004convex}, there must exist some dual-optimal solution $(\lambda, \hat{\bm{\alpha}}, \bm{\beta})$ which obeys complementarity with $\bm{X}^\star$. Moreover, $\vert X_{i,j}\vert \leq M_{i,j}$ is implied by $\mathrm{tr}(\bm{X})=1, \bm{X} \succeq \bm{0}$, while $\sum_{j=1}^p \vert X_{i,j}\vert \leq z_i \sqrt{k}$ is implied by $\vert X_{i,j}\vert
\leq M_{i,j}z_i$ and $\bm{e}^\top \bm{z} \leq k$. Therefore, by complementary slackness, we can take the constraints $\vert X_{i,j}\vert \leq M_{i,j}z_i$, {\color{black}$\sum_{j=1}^p \vert X_{i,j}\vert \leq z_i \sqrt{k}$} to be inactive when $z_i=1$ without loss of generality, which implies that $\hat{\alpha}_{i,j}^\star, {\color{black} \beta_i^\star}=0$ if $z_i=1$ in some dual-optimal solution. Moreover, we also have $\hat{\alpha}_{i,j}^\star=0$ if $z_j=1$, since $\hat{\bm{\alpha}}$ obeys the dual feasibility constraint $\lambda \mathbb{I}+\hat{\bm{\alpha}}\succeq \bm{\Sigma}$, and therefore is itself symmetric.

Next, observe that, by strong duality, $\lambda=\lambda_{\max}(\bm{\Sigma}_{1,1})$ in this dual-optimal solution, since $\bm{\alpha}$ only takes non-zero values if $z_i=z_j=0$ and does not contribute to the objective{\color{black}, and $\bm{\beta}$ is similar}.

Next, observe that, by strong duality and complementary slackness, any dual feasible $(\lambda, \hat{\bm{\alpha}}, {\color{black} \bm{\beta}})$ satisfying the above conditions is dual-optimal. Therefore, we need to find an $\hat{\bm{\alpha}}_{2,2}$ such that
\begin{align*}
     \begin{pmatrix}\lambda\mathbb{I} -\bm{\Sigma}_{1,1} & -\bm{\Sigma}_{1,2}\\ -\bm{\Sigma}_{2,1} & \lambda\mathbb{I}+\hat{\bm{\alpha}}_{2,2}-\bm{\Sigma}_{2,2}
    \end{pmatrix}\succeq \bm{0}.
\end{align*}
By the generalized Schur complement lemma \citep[see][Equation 2.41]{boyd1994linear}, this is PSD if and only if
\begin{enumerate}\itemsep0em
    \item $\lambda\mathbb{I} -\bm{\Sigma}_{1,1} \succeq \bm{0}$,
    \item $\left(\mathbb{I}-(\lambda\mathbb{I} -\bm{\Sigma}_{1,1})(\lambda\mathbb{I} -\bm{\Sigma}_{1,1})^\dag\right) \bm{\Sigma}_{1,2}=\bm{0}$, and
    \item $\lambda \mathbb{I}+\hat{\bm{\alpha}}_{2,2}-\bm{\Sigma}_{2,2}\succeq \bm{\Sigma}_{1,2}^\top \left(\lambda \mathbb{I}-\bm{\Sigma}_{1,1}\right)^\dag \bm{\Sigma}_{1,2}$.
\end{enumerate}
The first two conditions hold {because, as argued above, $\lambda$ is optimal and therefore feasible, and the conditions are independent of $\hat{\bm{\alpha}}_{2,2}$}. Therefore, it suffices to pick $\hat{\bm{\alpha}}_{2,2}$ in order that the third condition holds. We achieve this by setting $\hat{\bm{\alpha}}_{2,2}$ so the PSD constraint in condition (3) holds with equality.

Finally, let us {\color{black} optimize for} $\bm{\beta}$ to obtain stronger cuts (when $z_i=0$ we can pick any feasible $\beta_i$, but optimizing to set $\partial f(\bm{z})_i$ to be as small as possible gives stronger cuts). {\color{black} This is equivalent to solving the univariate minimization problem for each $\beta_i$:
\begin{align*}
    \min_{\beta_i} \left(\sum_{j=1}^p M_{i,j}\max(0, \vert \alpha_{i,j}\vert-\beta_i)+{\color{black}\sqrt{k}\beta_i}\right).
\end{align*}
Moreover, it is a standard result from max-$k$ optimization \citep[see, e.g.,][]{zakeri2014optimization, todd2018max} that
this is achieved by setting $\beta_i$ to be the $\lceil \sqrt{k}\ \rceil$ largest element of $\{{\alpha}_{i,j}\}_{j \in [p]} \cup \{\alpha_{i,i}\}$ in absolute magnitude, where we include $\alpha_{i,i}$ twice since $M_{i,i}=1$ while $M_{i,j}=1/2$ if $j \neq i$. }
\end{proof}

\subsection{Strengthening the Master Problem via the Gershgorin Circle Theorem} \label{sec:exact.circle}
To accelerate Algorithm \ref{alg:cuttingPlaneMethod}, 
we strengthen the master problem by imposing bounds from the circle theorem. Formally, we have the following result, which can be deduced from \citep[Theorem 6.1.1]{horn1990matrix}:
\begin{theorem}\label{thm:circletheorem}
For any vector $\bm{z} \in \{0, 1\}^p$ we have the following upper bound on $f(\bm{z})$
\begin{align}
    f(\bm{z}) \leq \max_{j \in [p]: z_j=1}\sum_{i \in [p]}z_i \vert \Sigma_{i,j}\vert.
\end{align}
\end{theorem}

Observe that this bound cannot be used to \textit{directly} strengthen Algorithm \ref{alg:cuttingPlaneMethod}'s master problem, since the bound is not convex in $\bm{z}$. Nonetheless, it can be successfully applied if we (a) impose a big-M assumption on Problem \eqref{OriginalSPCA}'s optimal objective and (b) introduce $p$ additional binary variables $\bm{s} \in \{0, 1\}^p: \bm{e}^\top \bm{s}=1$ {which model whether the $i$th Gershgorin disc is active; recall that each eigenvalue is contained in the union of the discs}. Formally, we impose the following valid inequalities in the master problem:
\begin{align}\label{eqn:gershgorincircle}
\exists \bm{s} \in \{0, 1\}^p: \ & \theta \leq \sum_{i \in [p]} z_i \vert \Sigma_{i,j}\vert+M(1-s_j) \ \forall j \in [p], \bm{e}^\top \bm{s}=1, \bm{s} \leq \bm{z},
\end{align}
{ where $\theta$ is the epigraph variable maximized in the master problem stated in Algorithm \ref{alg:cuttingPlaneMethod}{, and $M$ is an upper bound on the sum of the $k$ largest absolute entries in any column of $\bm{\Sigma}$.} Note that we set $\bm{s}\leq \bm{z}$ since if $z_i=0$ the $i$th column of $\bm{\Sigma}$ does not feature in the relevant submatrix of $\bm{\Sigma}$.}
In the above inequalities, a valid $M$ is given by any bound on the optimal objective. Since Theorem {\color{black}\ref{thm:circletheorem}} supplies one such bound for any given $\bm{z}$, we can compute \begin{align}
    M:=\max_{j \in [p]}\max_{\bm{z} \in \{0, 1\}^p: \bm{e}^\top \bm{z} \leq k} \ \sum_{i \in [p]}z_i \vert \Sigma_{i,j}\vert,
\end{align} {which can be done in $O(p^2)$ time.  }

{\color{black}To further improve Algorithm \ref{alg:cuttingPlaneMethod}, we also make use of the Gershgorin circle theorem before generating each outer-approximation cut. Namely, at a given node in a branch-and-bound tree, there are indices $i$ where $z_i$ has been fixed to $1$, indices $i$ where $z_i$ has been fixed to $0$, and indices $i$ where $z_i$ has not yet been fixed. Accordingly, we compute the worst-case Gershgorin bound—by taking the worst-case bound over each index $j$ such that $z_j$ has not yet been fixed to $0$, i.e., $$\max_{j: z_j\neq 0}\left\{\max_{\bm{s} \in \{0, 1\}^p: \bm{e}^\top \bm{s} \leq k}\left\{\sum_{i \in [p]}s_i \vert \Sigma_{i,j}\vert \ \text{s.t.} \ s_i=0\ \text{if}\ z_i=0, s_i=1 \ \text{if}\ z_i=1\right\}\right\}.$$ If this bound is larger than our incumbent solution then we generate an outer-approximation cut, otherwise the entire subtree rooted at this node does not contain an optimal solution and we use instruct the solver to avoid exploring this node via a \verb|callback|.}

{
Our numerical results in Section \ref{sec:numres} echo the empirical findings of \citet{berk2017} and indicate that Algorithm \ref{alg:cuttingPlaneMethod} performs substantially better when the Gershgorin bound is supplied in the master problem. Therefore, it is interesting to theoretically investigate the tightness, or at least the quality, of Gershgorin's bound. 
We supply some results in this direction in the following proposition:

\begin{proposition}\label{prop:gershgorinthmapprox}
Suppose that $\bm{\Sigma}$ is a scaled diagonally dominant matrix as defined by \cite{boman2005factor}, i.e., there exists some vector $\bm{d}>0$ such that $$d_i\Sigma_{i,i} \geq \sum_{j \in [p]: j \neq i}d_j\vert \Sigma_{i,j}\vert \ \forall i \in [p].$$ Then, letting $\rho:=\max_{i,j \in [p]} \{\frac{d_i}{d_j}\}$, the Gershgorin circle theorem provides a $(1+\rho)$-factor approximation, i.e., \begin{align}
    f(\bm{z}) \leq \max_{j \in [p]}\left\{\sum_{i \in [p]} z_i \vert \Sigma_{i,j}\vert \right\}\leq (1+\rho) f(\bm{z}) \quad \forall \bm{z} \in \{0, 1\}^p.
\end{align}
\end{proposition}
{\color{black}
\begin{remark}
Observe that, for a fixed $\bm{z}$, the ratio $\rho:=\max_{i,j \in [p]} \{\frac{d_i}{d_j}\}$ need only be computed over indices $i,j$ such that $z_i,z_j=1$. Moreover, for a partially specified $\bm{z}$—which might arise at an intermediate node in a branch-and-bound tree generated by Algorithm \ref{alg:cuttingPlaneMethod}—the ratio $\rho$ need only be computed over indices $i$ where $z_i$ is unspecified or set to $1$. This suggests that the quality of the Gershgorin bound improves upon branching.
\end{remark}
}
\begin{remark}
In particular, if $\bm{\Sigma} \in S^n_+$ is a diagonal matrix, then {\color{black}Equation \eqref{eqn:gershgorincircle}'s} bound is tight - which follows from the fact that the spectrum of $\bm{\Sigma}$ and the discs coincide if and only if $\bm{\Sigma}$ is diagonal \citep[see, e.g,][Chapter 6]{horn1990matrix}. Alternatively, if $\bm{\Sigma}$ is a diagonally dominant matrix then $\rho=1$ and the Gershgorin circle theorem provides a $2-$factor approximation.
\end{remark}
\begin{proof}
Scaled diagonally dominant matrices have scaled diagonally dominant principal minors—this is trivially true because $$d_i\Sigma_{i,i} \geq \sum_{j \in [p]: j \neq i}d_j\vert \Sigma_{i,j}\vert \ \forall i \in [p]\implies d_i\Sigma_{i,i} \geq \sum_{j \in [p]: j \neq i}d_j z_j\vert \Sigma_{i,j}\vert \ \forall i \in [p]: z_i=1$$for the same vector $\bm{d}>\bm{0}$ and therefore the following chain of inequalities holds
\begin{align*}
    f(\bm{z}) \leq & \max_{j \in [p]}\{\sum_{i \in [p]} z_i \vert \Sigma_{i,j}\vert \}=\max_{j \in [p]}\{z_j\Sigma_{j,j}+ \sum_{i \in [p]: j \neq i}z_i\vert \Sigma_{i,j}\vert \}\\
    & \leq \max_{j \in [p]}\{z_j\Sigma_{j,j}+\sum_{i \in [p]: j \neq i}\rho \frac{d_i}{d_j}z_i\vert \Sigma_{i,j}\vert\}\leq (1+\rho)\max_{j \in [p]}\{z_j\Sigma_{j,j}\}\leq (1+\rho) f(\bm{z}) \quad \forall \bm{z} \in \{0, 1\}^p, \end{align*}
where the second inequality follows because $\rho\geq \frac{d_i}{d_j}$, the third inequality follows from the scaled diagonal dominance of the principal submatrices of $\bm{\Sigma}$, and the fourth inequality holds because the leading eigenvalue of a PSD matrix is at least as large as each diagonal.
\end{proof}


}

To make clear the extent our numerical success depends upon Theorem \ref{thm:circletheorem}, our results in Section \ref{sec:numres} present implementations of Algorithm \ref{alg:cuttingPlaneMethod} both with and without the bound.

{
\subsection{Beyond Gershgorin: Further Strengthening via Brauer's Ovals of Cassini} \label{sec:exact.oval}
Given the relevance of Gershgorin's bound, we propose, in this section,  a stronger —yet more expensive to implement— upper bound, based on an generalization of the Gershgorin Circle theorem, namely Brauer's ovals of Cassini.

First, we derive a new upper-bound on $f(\bm{z})$ that is at least as strong as the one presented in Theorem \ref{thm:circletheorem} and often strictly stronger \citep[][Chapter 6]{horn1990matrix}:
\begin{theorem}\label{thm:cassini1}
For any vector $\bm{z} \in \{0, 1\}^p$, we have the following upper bound on $f(\bm{z})$:
\begin{align} \label{eqn:bound.ovals}
    f(\bm{z}) \leq  \max_{i,j \in [p]: i>j, z_i=z_j=1} \left\{\frac{\Sigma_{i,i}+\Sigma_{j,j}}{2}+\frac{\sqrt{(\Sigma_{i,i}-\Sigma_{j,j})^2+4R_i(\bm{z}) R_j(\bm{z})}}{2}\right\},
\end{align}
where $R_i(\bm{z}):=\sum_{j \in [p]: j \neq i} z_j \vert \Sigma_{i,j}\vert$ is the absolute sum of off-diagonal entries in the $i$th column of the submatrix of $\bm{\Sigma}$ induced by $\bm{z}$.
\end{theorem}
\begin{proof}
Let us first recall that, per \citet{brauer1946limits}'s original result, all eigenvalues of a matrix $\bm{\Sigma} \in S^p_+$ are contained in the union of the following $p(p-1)/2$ ovals of Cassini:
\begin{align*}
    \bigcup_{i \in [p], j \in [p]: i < j} \left\{\lambda \in \mathbb{R}_+: \vert \lambda-\Sigma_{i,i}\vert \vert \lambda-\Sigma_{j,j}\vert \leq R_i R_j \right\},
\end{align*}
where $R_i:=\sum_{j \in [p]: j \neq i} \vert \Sigma_{i,j}\vert$ is the absolute sum of off-diagonal entries in the $i$th column of $\bm{\Sigma}$. Next, let us observe that, if $\lambda$ is a dominant eigenvalue of a PSD matrix $\bm{\Sigma}$ then $\lambda \geq \Sigma_{i,i} \ \forall i $ and, in the $(i,j)$th oval, the bound reduces to
\begin{align}\label{eqn:dominanteigenvaluebound}
    \lambda^2 -\lambda(\Sigma_{i,i}+\Sigma_{j,j})+\Sigma_{i,i}\Sigma_{j,j}-R_i R_j \leq 0,
\end{align}
which, by the quadratic formula, implies an upper bound is $\frac{\Sigma_{i,i}+\Sigma_{j,j}}{2}+\frac{\sqrt{(\Sigma_{i,i}-\Sigma_{j,j})^2+4R_i R_j}}{2}$. The result follows because if $z_i=0$ the $i$th row of $\bm{\Sigma}$ cannot be used to bound $f(\bm{z})$.
\end{proof}

Theorem \ref{thm:cassini1}'s inequality can be enforced numerically as mixed-integer second order cone constraints. Indeed, the square root term in \eqref{eqn:bound.ovals} can be modeled using second-order cone, and the bilinear terms only involve binary variables and can be linearized.
Completing the square in Equation \eqref{eqn:dominanteigenvaluebound}, \eqref{eqn:bound.ovals} is equivalent to the following system of $p(p-1)/2$ mixed-integer second-order cone inequalities:
\begin{align*}
    \left(\theta-\frac{1}{2}(\Sigma_{i,i}+\Sigma_{j,j})\right)^2 \leq \sum_{s,t \in [p]: s\neq i, t \neq j} W_{s,t}\vert \Sigma_{i,s}\Sigma_{j,t}\vert -\frac{3}{4}\Sigma_{i,i}\Sigma_{j,j}+M(1-s_{i,j}) \ \forall i,j \in [p]: i <j,\\
    \sum_{i,j \in [p]: i<j} s_{i,j}=1, s_{i,j} \leq \min(z_i, z_j) \ i,j \in [p]: i<j, \ s_{i,j} \in \{0, 1\} \ i,j \in [p]: i<j.
\end{align*}
where $W_{i,j}=z_i z_j$ is a product of binary variables which can be modeled using, e.g., the {\color{black}McCormick} inequalities $\max(0, z_i+z_j-1) \leq W_{i,j} \leq \min(z_i, z_j)$, and $M$ is an upper bound on the right-hand-side of the inequality for any $i,j: i \neq j$, which can be computed in $O(p^3)$ time in much the same manner as a big-$M$ constant was computed in the previous section. Note that we do not make use of these inequalities directly in our numerical experiments, due to their high computational cost. However, an interesting extension would be to introduce the binary
variables dynamically, via branch-and-cut-and-price \citep{barnhart1998branch}.

Since the bound derived from the ovals of Cassini (Theorem  \ref{thm:cassini1}) is at least as strong as the Gershgorin circle's one (Theorem \ref{thm:circletheorem}), it satisfies the same approximation guarantee (Proposition \ref{prop:gershgorinthmapprox}). In particular, it is tight when $\bm{\Sigma}$ is diagonal and provides a $2-$factor approximation for diagonally dominant matrices.
Actually, we now prove a stronger result and demonstrate that Theorem \ref{thm:cassini1} provides a $2-$factor bound on $f(\bm{z})$ for doubly diagonally dominant matrices—a broader class of matrices than diagonally dominant matrices \citep[see][for a general theory]{li1997doubly}:
\begin{proposition}
Let $\bm{\Sigma}\in S^p_+$ be a doubly diagonally dominant matrix, i.e.,
\begin{align*}
    \Sigma_{i,i}\Sigma_{j,j}\geq R_i R_j \ \forall i,j \in [p]: i >j,
\end{align*}
where $R_i:=\sum_{j \in [p]: j \neq i} \vert \Sigma_{i,j} \vert$ is the sum of the off-diagonal entries in the $i$th column of $\bm{\Sigma}$. Then, we have that
\begin{align}
    f(\bm{z}) \leq  \max_{i,j \in [p]: i>j, z_i=z_j=1} \left\{\frac{\Sigma_{i,i}+\Sigma_{j,j}}{2}+\frac{\sqrt{(\Sigma_{i,i}-\Sigma_{j,j})^2+4R_i(\bm{z}) R_j(\bm{z})}\}}{2}\right\} \leq 2f(\bm{z}).
\end{align}
\end{proposition}
\begin{proof}
Observe that if $\Sigma_{i,i}\Sigma_{j,j}\geq R_i R_j$ then $$\sqrt{(\Sigma_{i,i}-\Sigma_{j,j})^2+4 R_i R_j}\leq \sqrt{(\Sigma_{i,i}-\Sigma_{j,j})^2+4 \Sigma_{i,i}\Sigma_{j,j}}=\Sigma_{i,i}+\Sigma_{j,j}.$$ The result then follows in essentially the same fashion as Proposition \ref{prop:gershgorinthmapprox}.
\end{proof}

}

\section{Convex Relaxations and Rounding Methods}\label{sec:relaxandround}
For large-scale instances, high-quality solutions can be obtained by solving a convex relaxation of Problem  \eqref{misdpprimal} and rounding the optimal solution. In Section \ref{ssec:relax.bool}, we propose relaxing $\bm{z} \in \{0, 1\}^p$ in \eqref{misdpprimal} to $\bm{z} \in [0, 1]^p$ and applying a greedy rounding scheme. We further tighten this relaxation using second-order cones constraints in Section \ref{ssec:validineq}.

\subsection{A Boolean Relaxation and a Greedy Rounding Method} \label{ssec:relax.bool}
We first consider a Boolean relaxation of \eqref{misdpprimal}, which we obtain\footnote{\color{black}Note that we omit the {\color{black}$\sum_{j=1}^p \vert X_{i,j}\vert \leq \sqrt{k} z_i$} constraints when we develop our convex relaxations, since they are essentially dominated by the $\Vert \bm{X}\Vert_1 \leq k$ constraint we introduce in the next section; we introduced these inequalities to improve our semidefinite-free subproblem strategy for the exact method.} by relaxing $\bm{z} \in \{0, 1\}^p$ to $\bm{z} \in [0, 1]^p$. This gives $\displaystyle \max_{\bm{z} \in [0, 1]^p: \bm{e}^\top \bm{z} \leq k} \: f(\bm{z})$, {\color{black}i.e.},
\begin{equation}
    \begin{aligned}\label{prob:lprelax2}
    \max_{\substack{\bm{z} \in [0, 1]^p}: \bm{e}^\top \bm{z} \leq k} \: \max_{\bm{X} \succeq \bm{0}} \quad & \langle \bm{\Sigma}, \bm{X} \rangle \ \text{s.t.} \ \mathrm{tr}(\bm{X})=1, \vert X_{i,j}\vert \leq M_{i,j}z_i\ \forall i,j \in [p].
\end{aligned}
\end{equation}
A useful strategy for obtaining a high-quality feasible solution is to solve \eqref{prob:lprelax2} and set $z_i=1$ for $k$ indices corresponding to the largest $\bm{z}_j$'s in \eqref{prob:lprelax2} {\color{black}as proposed in the randomized case for general integer optimization problems by \cite{raghavan1987randomized}}. We formalize this in Algorithm \ref{alg:greedymethod}. {\color{black}We remark that rounding strategies for sparse PCA have previously been proposed \citep[see][]{fountoulakis2017randomized, dey2017sparse, chowdhury2020approximation}, however, the idea of rounding $\bm{z}$ and then optimizing for $\bm{X}$ appears to be new.}
\begin{algorithm*}
\caption{A greedy rounding method for Problem \eqref{OriginalSPCA}}
\label{alg:greedymethod}
\begin{algorithmic}\normalsize
\REQUIRE Covariance matrix $\bm{\Sigma}$, sparsity parameter $k$
\STATE Compute $\bm{z}^\star$ solution of \eqref{prob:lprelax2} or \eqref{spca:sdpplussocp}
\STATE Construct $\bm{z} \in \{0, 1\}^p: \bm{e}^\top \bm{z}=k$ such that $z_i\geq z_j$ if $z^\star_i \geq z^\star_j$.
\STATE Compute $\bm{X}$ solution of \vspace{-2mm}
\begin{align*}
    \max_{\bm{X} \in S^p_+} \ \langle \bm{\Sigma}, \bm{X} \rangle \ \text{s.t.} \ \mathrm{tr}(\bm{X})=1, X_{i,j}=0 \ \text{if} \ z_i z_j=0 \ \forall i,j \in [p].
\end{align*}  \vspace{-5mm}
\RETURN $\bm{z}, \bm{X}$.
\end{algorithmic}
\end{algorithm*}

{
\begin{remark}\label{remark:scalability}
Our numerical results in Section \ref{sec:numres} reveal that explicitly imposing a PSD constraint on $\bm{X}$ in the relaxation \eqref{prob:lprelax2}—or the ones derived later in the following section—prevents our approximation algorithm from scaling to larger problem sizes than the exact Algorithm \ref{alg:cuttingPlaneMethod} can already solve. Therefore, to improve scalability, the semidefinite cone can be safely approximated via its second-order cone relaxation, $X_{i,j}^2 \leq X_{i,i}X_{j,j}\ \forall i,j \in [p]$, plus a small number of cuts of the form $\langle \bm{X}, \bm{x}_t\bm{x}_t^\top\rangle \geq 0$ as presented in \citet{bertsimas2019polyhedral}.
\end{remark}
}

{
\begin{remark}
Rather than relaxing and greedily rounding $\bm{z}$, one could consider a higher dimensional relax-and-round scheme where we let $\bm{Z}$ model the outer product $\bm{z}\bm{z}^\top$ via $\bm{Z}\succeq \bm{z}\bm{z}^\top$, $\max(0, z_i+z_j-1)\leq Z_{i,j} \leq \min(z_i, z_j) \ \forall i,j \in [p]$, $Z_{i,i}=z_i$, and require that $\sum_{i,j \in [p]}Z_{i,j} \leq k^2$. Indeed, a natural ``round'' component of such a relax-and-round scheme is precisely Goemans-Williamson rounding \citep[][]{goemans1995improved,bertsimas1998semidefinite}, which performs at least as well as greedy rounding in both theory and practice. Unfortunately, some preliminary numerical experiments indicated that Goemans-Williamson rounding is not actually much better than greedy rounding in practice, and is considerably more expensive to implement. Therefore, we defer the details of the Goemans-Williamson scheme to Appendix \ref{sec:relax.goemans}, and do not consider it any further in this paper.
\end{remark}

}

\subsection{Valid Inequalities for Strengthening Convex Relaxations}\label{ssec:validineq}
We now propose valid inequalities which allow us to improve the quality of the convex relaxations discussed previously. 
Note that as convex relaxations and random rounding methods are two sides of the same coin \citep{barak2014rounding}, applying these valid inequalities also improves the quality of the randomly rounded solutions. 

\begin{theorem}
Let $\mathcal{P}_{strong}$ denote the optimal objective value of the following problem:
\begin{equation}
\label{spca:sdpplussocp}
\begin{aligned}
    \max_{\bm{z} \in [0, 1]^p: \bm{e}^\top \bm{z} \leq k} \max_{\bm{X} \in S^p_+} \ \langle \bm{\Sigma}, \bm{X} \rangle \ \text{s.t.} \quad & \mathrm{tr}(\bm{X})=1,
    \vert X_{i,j}\vert \leq M_{i,j}z_i\ \forall i,j \in [p] ,\\
    &\sum_{j \in [p]} X_{i,j}^2 \leq X_{i,i}z_i, \Vert \bm{X}\Vert_1 \leq k.
\end{aligned}
\end{equation}
Then, \eqref{spca:sdpplussocp} is a stronger relaxation than \eqref{prob:lprelax2}, i.e., the following inequalities hold:
\begin{align*}
    \max_{\bm{z} \in [0, 1]^p: \bm{e}^\top \bm{z} \leq k} f(\bm{z})\geq \mathcal{P}_{strong} \geq \max_{\bm{z} \in \{0, 1\}^p: \bm{e}^\top \bm{z} \leq k} f(\bm{z}).
\end{align*}
Moreover, suppose an optimal solution to \eqref{spca:sdpplussocp} is of rank one. Then, the relaxation is tight: $$\mathcal{P}_{strong}= \max_{\bm{z} \in \{0, 1\}^p: \bm{e}^\top \bm{z} \leq k} f(\bm{z}).$$
\end{theorem}

\begin{proof}
The first inequality $\max_{\bm{z} \in [0, 1]^p: \bm{e}^\top \bm{z} \leq k} f(\bm{z})\geq \mathcal{P}_{strong}$ is trivial. The second inequality holds because $\mathcal{P}_{strong}$ is indeed a valid relaxation of Problem \eqref{OriginalSPCA}. Indeed, $\| \bm{X} \|_1 \leq k$ follows from the cardinality and big-M constraints. The semidefinite constraint $\bm{X} \succeq 0$ impose second-order cone constraints on the $2\times 2$ minors of $\bm{X}$, $X_{i,j}^2 \leq z_i X_{i,i} X_{j,j}$, which can be aggregated into $\sum_{j \in [p]} X_{i,j}^2 \leq X_{i,i} z_i $ \citep[see][for derivations]{bertsimas2019polyhedral}.

Finally, suppose that an optimal solution to Problem \eqref{spca:sdpplussocp} is of rank one, i.e., the optimal matrix $\bm{X}$ can be decomposed as $\bm{X}=\bm{x}\bm{x}^\top$. Then, the SOCP inequalities imply that $\sum_{j \in [p]} x_i^2 x_j^2 \leq x_i^2 z_i$. However, $\sum_{j \in [p]}x_j^2=\mathrm{tr}(\bm{X})=1$, which implies that $x_i^2 \leq x_i^2 z_i$, i.e., $z_i=1$ for any index $i$ such that $\vert x_i\vert>0$. Since $\bm{e}^\top \bm{z} \leq k$, this implies that $\Vert \bm{x}\Vert_0 \leq k$, i.e., $\bm{X}$ also solves Problem \eqref{sdospca1}.
\end{proof}

{ As our numerical experiments will demonstrate and despite the simplicity of our rounding mechanism in Algorithm \ref{alg:greedymethod}, the relaxation \eqref{spca:sdpplussocp} provides high-quality solutions to the original sparse PCA problem \eqref{OriginalSPCA}, without introducing any additional variables.} {\color{black}We remark that other inequalities, including the second-order cone inequalities proposed in \citet[Lemma 2 (ii)]{li2020exact}, could further improve the convex relaxation; we leave integrating these inequalities within our framework as future work.}

\section{Numerical Results}\label{sec:numres}
We now assess the numerical behavior of the algorithms proposed in Section \ref{sec:reformulation} and \ref{sec:relaxandround}. To bridge the gap between theory and practice, we present a \verb|Julia| code which implements the described convex relaxation and greedy rounding procedure on GitHub\footnote{\href{github.com/ryancorywright/ScalableSPCA.jl}{https://github.com/ryancorywright/ScalableSPCA.jl}}. The code requires a conic solver such as \verb|Mosek| and several open source Julia packages to be installed.
\subsection{Performance of Exact Methods}
In this section, we apply Algorithm \ref{alg:cuttingPlaneMethod} to medium and large-scale sparse principal component analysis problems, with and without Gershgorin circle theorem bounds in the master problem. All experiments were implemented in \verb|Julia| $1.3$, using \verb|Gurobi| $9.1$ and \verb|JuMP.jl| $0.21.6$, and performed on a standard Macbook Pro laptop, with a $2.9$GHz $6$-Core Intel i9 CPU, using $16$ GB DDR4 RAM. We compare our approach to the branch-and-bound algorithm\footnote{\color{black}The solve times for their method, as reported here, differ from those reported in \citet{berk2017} due to a small typo in their implementation (line $110$ of their branchAndBound.jl code should read ``if $y[i]==-1$ $||$ $y[i]==1$'', not ``if $y[i]==-1$'' in order to correctly compute the Gershgorin circle theorem bound); correcting this is necessary to ensure that we obtain correct results from their method.} developed by \cite{berk2017} on the UCI \verb|pitprops|, \verb|wine|, \verb|miniboone|, \verb|communities|, \verb|arrythmia| and \verb|micromass| datasets, both in terms of runtime and the number of nodes expanded; we refer to \cite{berk2017, bertsimas2019polyhedral} for descriptions of these datasets. Note that we normalized all datasets before running the method (i.e., we compute the leading sparse principal components of correlation matrices). Additionally, we warm-start all methods with the solution from 
the method of \cite{yuan2013truncated}, to maintain a fair comparison.

Table \ref{tab:comparison} reports the time for Algorithm \ref{alg:cuttingPlaneMethod} (with and without Gershgorin circle theorem bounds in the master problem) and the method of \cite{berk2017} to identify the leading $k$-sparse principal component for {\color{black}$k \in \{5, 10, 20\}$}, along with the number of nodes expanded, and the number of outer approximation cuts generated. We impose a relative optimality tolerance of $10^{-3}$ for all approaches {\color{black}, i.e., terminate each method when $(UB-LB)/UB\leq 10^{-3}$ where $UB$ denotes the current objective bound and $LB$ denotes the current incumbent objective value}. Note that $p$ denotes the dimensionality of the correlation matrix, and $k \leq p$ denotes the target sparsity.

{
\begin{table}[h]
\centering\footnotesize
\caption{{\color{black}Runtime in seconds per approach. We impose a time limit of $600$s. If a solver fails to converge, we report the relative bound gap at termination in brackets.
}}
\begin{tabular}{@{}l l l r r r r r r r r r@{}} \toprule
Dataset &  $p$ & $k$ & \multicolumn{3}{c@{\hspace{0mm}}}{Alg. \ref{alg:cuttingPlaneMethod}} &  \multicolumn{3}{c@{\hspace{0mm}}}{Alg. \ref{alg:cuttingPlaneMethod}+ Circle Theorem} &  \multicolumn{2}{c@{\hspace{0mm}}}{Method of B.+B.} \\
\cmidrule(l){4-6} \cmidrule(l){7-9} \cmidrule(l){10-11} & &  & Time(s) & Nodes & Cuts & Time(s) & Nodes & Cuts & Time(s) & Nodes \\\midrule
 Pitprops & $13$ & $5$   & \color{black} $0.30$ & \color{black} $1,608$ & \color{black} $1,176$ & \color{black} $\textbf{0.06}$ & \color{black} $38$ & \color{black} $8$ & $1.49$ & $22$ \\
&      & $10$  & \color{black} $0.14$ &  \color{black} $414$ & \color{black} $387$ &  \color{black} $\textbf{0.02}$ &  \color{black} $18$ &  \color{black} $21$ & $\textbf{0.02}$ & $14$ \\\midrule
Wine & $13$ & $5$   & \color{black} $0.57$ & \color{black} $2,313$ & \color{black} $1,646$ & \color{black} $\textbf{0.02}$ & \color{black} $46$ & \color{black} $11$ & $0.04$ & $34$ \\
         &      & $10$  & \color{black} $0.17$ & \color{black} $376$ & \color{black} $311$ & $\color{black} 0.03$ & \color{black} $54$ & \color{black} $58$ & $\textbf{0.02}$ & $12$ \\\midrule
Miniboone & $50$ & $5$   & \color{black} $\textbf{0.01}$ & \color{black} $0$ & \color{black} $11$ & \color{black} $\textbf{0.01}$ & \color{black} $0$ & \color{black} $3$ & $0.04$ & $2$ \\
 &  & $10$   & \color{black} $\textbf{0.01}$ & \color{black} $0$ & \color{black} $16$ & \color{black} $0.02$ & \color{black} $0$ & \color{black} $3$ & $0.04$ & $2$ \\
 & &  \color{black} $20$ &  \color{black} $0.03$ &  \color{black} $0$ &  \color{black} $26$ &  \color{black} $\textbf{0.01}$ &  \color{black} $0$ &  \color{black} $3$ &  \color{black} $1.30$ &  \color{black} $5,480$ \\\midrule
 Communities & $101$ & $5$ & \color{black} ($2.87\%$) & \color{black} $28,462$ & \color{black} $25,483$ & \color{black} $\textbf{0.20}$ & \color{black} $201$ & \color{black} $3$ & $0.57$ & $101$ \\
   &  & $10$ & \color{black}($13.3\%$) & \color{black} $37,479$ & \color{black} $36,251$ & \color{black} $\textbf{0.34}$ & \color{black} $406$ & \color{black} $39$ & $0.94$ & $1,298$ \\
& &  \color{black} $20$ &  \color{black} ($39.6\%)$ &  \color{black} $24,566$ &  \color{black} $24,632$ &  \color{black} ($12.1\%)$ &  \color{black} $42,120$ &  \color{black} $37,383$  &  \color{black} $(\textbf{9.97\%})$ &  \color{black} $669,500$ \\\midrule
 Arrhythmia & $274$ & $5$ & \color{black} ($18.1\%$) & \color{black} $22,771$ & \color{black} $20,722$ & \color{black}$6.07$ & \color{black} $135$ & \color{black} $1,233$ & $\textbf{4.17}$ & $1,469$ \\
   & & $10$ & \color{black} ($32.6\%)$ & \color{black} $19,500$ & \color{black} $19,314$ & \color{black} ($2.92\%$) & \color{black} $15,510$ & \color{black} $6,977$ & $\textbf{(0.83\%)}$ & $471,680$ \\
    & &  \color{black} $20$ &  \color{black} $(74.4\%)$ &  \color{black} $33,773$ &  \color{black} $12,374$&  \color{black} ($24.3\%$) &  \color{black} $33,123$ &  \color{black} $19,662$ &  \color{black} $(\textbf{18.45\%})$ &  \color{black} $311,400$ \\\midrule
Micromass & $1300$ & $5$ & \color{black} $(1.29\%)$ & \color{black} $3,859$ & \color{black} $3,099$ & \color{black} $163.60$ & \color{black} $2,738$ & \color{black} $6$ & $\textbf{24.31}$ & $1,096$ \\
   & & $10$ & \color{black}  $(10.6\%)$ & \color{black}  $3,366$ & \color{black}  $3,369$ & \color{black} $\textbf{241.86}$ & \color{black} $3,233$ & \color{black} $121$ & $362.4$ & $36,690$ \\
    & &  \color{black}  \color{black} $20$ &  \color{black} ($35.9\%)$ &  \color{black} $2,797$ &  \color{black} $2,839$ &  \color{black} ($35.9\%$) &  \color{black} $2,676$ &  \color{black} $2,115$ &  \color{black} $(\textbf{10.34}\%)$ &  \color{black} $31,990$ \\
\bottomrule
\end{tabular}
\label{tab:comparison}
\end{table}
}

Our main findings from these experiments are as follows:
\begin{itemize}\setlength\itemsep{0em}
    \item For smaller problems, the strength of Algorithm \ref{alg:cuttingPlaneMethod}'s cuts allows it to outperform state-of-the-art methods such as the method of \cite{berk2017}. Moreover, for larger problem sizes, the adaptive branching strategy {\color{black}performs comparably to} Algorithm \ref{alg:cuttingPlaneMethod}. This suggests that {\color{black}the relative merits of both approaches are roughly even, and which method is preferable may depend on the problem data.}
    \item Generating outer-approximation cuts and valid upper bounds from the Gershgorin circle theorem are both powerful ideas, but the greatest aggregate power appears to arise from intersecting these bounds, rather than using one bound alone.
\end{itemize}
\begin{itemize}
    \item {\color{black}Once both $k$ and $p$ are sufficiently large (e.g. $p>300$ and $k>10$), no approach is able to solve the problem to provable optimality within $600$s. This motivates our study of convex relaxations and randomized rounding methods in the next section.}
\end{itemize}
\vspace{-5mm}

\subsection{Convex Relaxations and Randomized Rounding Methods}
In this section, we apply Algorithm \ref{alg:greedymethod} to obtain high quality convex relaxations and feasible solutions for the datasets studied in the previous subsection, and compare the relaxation to a difference convex relaxation developed by \citet{d2008optimal}, in terms of the quality of the upper bound and the resulting greedily rounded solutions. All experiments were implemented using the same specifications as the previous section. { Note that \citet{d2008optimal}'s upper bound\footnote{ Strictly speaking, \citet{d2008optimal} does not actually write down this formulation in their work. Indeed, their bound involves dual variables which cannot be used directly to generate feasible solutions via greedy rounding. However, the fact that this bound and \citep[Problem (8)]{d2008optimal} are dual to each other follows directly from strong semidefinite duality, and therefore we refer to this formulation as being due to \cite{d2008optimal} (it essentially is).} which we compare against is:
\begin{equation}
    \begin{aligned}\label{prob:lprelax3}
    \max_{\substack{\bm{z} \in [0, 1]^p}: \bm{e}^\top \bm{z} \leq k} \: \max_{\bm{X} \succeq \bm{0}, \bm{P}_i \succeq \bm{0}\ \forall i \in [p]} \quad & \sum_{i \in [p]}\ \langle \bm{a}_i\bm{a}_i^\top, \bm{P}_i\rangle \ \text{s.t.} \ \mathrm{tr}(\bm{X})=1,\ \mathrm{tr}(\bm{P}_i)=z_i, \ \bm{X}\succeq \bm{P}_i \ \forall i \in [p],
\end{aligned}
\end{equation}
where $\bm{\Sigma}=\sum_{i=1}^p \bm{a}_i\bm{a}_i^\top$ is a Cholesky decomposition of $\bm{\Sigma}$, and we obtain feasible solutions from this relaxation by greedily rounding an optimal $\bm{z}$ in the bound \textit{\`{a} la} Algorithm \ref{alg:greedymethod}. {\color{black} To allow for a fair comparison,} we also consider augmenting this formulation with the inequalities derived in Section \ref{ssec:validineq} to obtain the following stronger yet more expensive to solve relaxation:
\begin{equation}
   \label{spca:sdpplussocp2}
    \begin{aligned}
    \max_{\substack{\bm{z} \in [0, 1]^p}: \bm{e}^\top \bm{z} \leq k} \: \max_{\substack{\bm{X} \succeq \bm{0},\\ \bm{P}_i \succeq \bm{0}\ \forall i \in [p]}} \quad & \sum_{i \in [p]}\ \langle \bm{a}_i\bm{a}_i^\top, \bm{P}_i\rangle \ \text{s.t.} \ \mathrm{tr}(\bm{X})=1,\ \mathrm{tr}(\bm{P}_i)=z_i, \ \bm{X}\succeq \bm{P}_i \ \forall i \in [p],\\
    &\sum_{j \in [p]} X_{i,j}^2 \leq X_{i,i}z_i, \Vert \bm{X}\Vert_1 \leq k.
\end{aligned}
\end{equation}
}

{\color{black} We first apply these relaxations on datasets where Algorithm \ref{alg:cuttingPlaneMethod} terminates, hence the optimal solution is known and can be compared against.} We report the quality of both methods with and without the additional inequalities discussed in Section \ref{ssec:validineq}, in Tables \ref{tab:comparison_convrelaxations}-\ref{tab:comparison_convrelaxations2} respectively\footnote{For the instances of \eqref{prob:lprelax3} or \eqref{spca:sdpplussocp2} where $p>13$ we used SCS version $2.1.1$ (with default parameters) instead of Mosek, since Mosek required more memory than was available in our computing environment, and SCS takes an augmented Lagrangian approach which is less numerically stable but requires significantly less memory. That is, \eqref{prob:lprelax3}'s formulation is too expensive to solve via IPMs on a laptop when $p=50$.}.

\begin{table}[h]
\centering\footnotesize
\caption{{\color{black}Quality of relaxation gap (upper bound vs. optimal solution-denoted R. gap), objective gap (rounded solution vs. optimal solution-denoted O. gap) and runtime in seconds per method.}}
\begin{tabular}{@{}l l l r r r r r r@{}} \toprule
Dataset &  $p$ & $k$ & \multicolumn{3}{c@{\hspace{0mm}}}{Alg. \ref{alg:greedymethod} with \eqref{prob:lprelax2}} & \multicolumn{3}{c@{\hspace{0mm}}}{Alg. \ref{alg:greedymethod} with \eqref{prob:lprelax3}} \\
\cmidrule(l){4-6} \cmidrule(l){7-9} & &  & R. gap $(\%)$& O. gap  $(\%)$ & Time(s) & R. gap  $(\%)$ & O. gap $(\%)$ & Time(s) \\\midrule
Pitprops & $13$ & $5$ &$23.8\%$ & $0.00\%$ &$0.02$ &$23.8\%$ & $16.1\%$ &$0.46$\\
 &  & $10$ & $1.10\%$ &$0.30\%$ &$0.03$ &$1.10\%$ & $1.33\%$ &$0.46$\\\midrule
Wine & $13$ & $5$ &$36.8\%$ & $0.00\%$ &$0.02$ &$36.8\%$ & $40.4\%$ &$0.433$\\
 &  & $10$ &$2.43\%$ & $0.26\%$ &$0.03$ &$2.43\%$ & $15.0\%$ &$0.463$\\\midrule
 Miniboone & $50$ & $5$ & $781.3\%$ & $235.6\%$ & $7.37$ & $781.2\%$ & $34.7\%$ & $1,191.0$\\
& & $10$ & $340.6\%$ & $117.6\%$ & $7.50$ & $340.6\%$ & $44.9\%$ & $1,102.6$\\
& & \color{black} $20$ &  \color{black} $120.3\%$ &  \color{black} $38.08\%$ &  \color{black} $6.25$ &  \color{black} $120.3\%$ &  \color{black} $31.9\%$ &  \color{black} $1,140.2$ \\
\bottomrule
\end{tabular}
\label{tab:comparison_convrelaxations}
\end{table}

\begin{table}[h]
\centering\footnotesize
\caption{\color{black} Quality of relaxation gap (upper bound vs. optimal solution-denoted R. gap), objective gap (rounded solution vs. optimal solution-denoted O. gap) and runtime in seconds per method, with additional inequalities from Section \ref{ssec:validineq}.}
\begin{tabular}{@{}l l l r r r r r r@{}} \toprule
Dataset &  $p$ & $k$ & \multicolumn{3}{c@{\hspace{0mm}}}{Alg. \ref{alg:greedymethod} with \eqref{spca:sdpplussocp}} &  \multicolumn{3}{c@{\hspace{0mm}}}{Alg. \ref{alg:greedymethod} with \eqref{spca:sdpplussocp2}} \\
\cmidrule(l){4-6} \cmidrule(l){7-9} & &  & R. gap $(\%)$& O. gap  $(\%)$ & Time(s) & R. gap  $(\%)$ & O. gap $(\%)$ & Time(s) \\\midrule
Pitprops & $13$ & $5$ & $0.71\%$ &$0.00\%$ & $0.17$ &$1.53\%$ & $0.00\%$ &$0.55$ \\
 &  & $10$ & $0.12\%$ &$0.00\%$ & $0.27$ &$1.10\%$ & $0.00\%$ &$3.27$\\\midrule
Wine & $13$ & $5$ &$1.56\%$ & $0.00\%$ &$0.24$ &$2.98\%$ & $15.03\%$ &$0.95$\\
 &  & $10$ &$0.40\%$ & $0.00\%$ &$0.22$ &$2.04\%$ & $0.00\%$ &$1.15$\\\midrule
Miniboone & $50$ & $5$ & $0.00\%$ & $0.00\%$ & $163.3$ & $0.00\%$ & $0.01\%$ & $500.7$\\
& & $10$ & $0.00\%$ & $0.00\%$ & $148.5$ & $0.00\%$ & $0.02\%$ & $489.9$\\
& &  \color{black} $20$ &  \color{black} $0.00\%$ &  \color{black} $0.00\%$ &  \color{black} $194.5$ &  \color{black} $0.00\%$ &  \color{black} $0.00\%$ &  \color{black} $776.3$\\
\bottomrule
\end{tabular}
\label{tab:comparison_convrelaxations2}
\end{table}

Observe that applying Algorithm \ref{alg:greedymethod} without the additional inequalities (Table \ref{tab:comparison_convrelaxations}) yields rather poor relaxations and randomly rounded solutions. However, by intersecting our relaxations with the additional inequalities from Section \ref{ssec:validineq} (Table \ref{tab:comparison_convrelaxations2}), we obtain extremely high quality relaxations. Indeed, with the additional inequalities, Algorithm \ref{alg:greedymethod} using formulation \eqref{spca:sdpplussocp} identifies the optimal solution in all instances (0\% O. gap), and always supplies a bound gap of less than $2\%$. Moreover, in terms of obtaining high-quality solutions, the new inequalites allow Problem \eqref{spca:sdpplussocp} to perform
as well or better as Problem \eqref{prob:lprelax3}, despite optimizing over one semidefinite matrix, rather than $p+1$ semidefinite matrices. This suggests that Problem \eqref{spca:sdpplussocp} should be considered as a viable, more scalable and more accurate alternative to existing SDO relaxations such as Problem \eqref{prob:lprelax3}. For this reason, we shall only consider using Problem \eqref{spca:sdpplussocp}'s formulation for the rest of the paper.

We remark however that the key drawback of applying these methods is that, as implemented in this section, they do not scale to sizes beyond which Algorithm \ref{alg:cuttingPlaneMethod} successfully solves. This is a drawback because Algorithm \ref{alg:cuttingPlaneMethod} supplies an exact certificate of optimality, while these methods do not. In the following set of experiments, we therefore investigate numerical techniques to improve the scalability of Algorithm \ref{alg:greedymethod}.

\subsection{Scalable Dual Bounds and {\color{black}Randomized} Rounding Methods}
To improve the scalability of Algorithm \ref{alg:greedymethod}, we relax the PSD constraint on $\bm{X}$ in \eqref{prob:lprelax2} and \eqref{spca:sdpplussocp}. With these enhancements, we demonstrate that Algorithm \ref{alg:greedymethod} can be successfully scaled to generate high-quality bounds for $1000s \times 1000s$ matrices.

{ As discussed in Remark \ref{remark:scalability}, we can replace the PSD constraint $\bm{X} \succeq \bm{0}$ by requiring that the $p(p-1)/2$ two by two minors of $\bm{X}$ are non-negative: $X_{i,j}^2 \leq X_{i,i} X_{j,j}$. Second, we consider adding $20$ linear inequalities of the form $\langle \bm{X}, \bm{x}_t\bm{x}_t^\top\rangle \geq 0$, for some vector $\bm{x}_t$ \citep[see][for a discussion]{bertsimas2019polyhedral}.} Table \ref{tab:comparison_convrelaxations3} reports the performance of Algorithm \ref{alg:greedymethod} (with the relaxation \eqref{spca:sdpplussocp}) with these two approximations of the positive semidefinite cone, ``Minors'' and ``Minors + 20 inequalities'' respectively. {\color{black}Note that we report the entire duality gap (i.e. do not break the gap down into its relaxation and objective gap components) since, as reflected in Table \ref{tab:comparison}, some of these instances are currently too large to solve to optimality.}

\begin{table}[h]
\centering\footnotesize
\caption{Quality of bound gap (rounded solution vs. upper bound) and runtime in seconds of Algorithm \ref{alg:greedymethod} with \eqref{spca:sdpplussocp}, outer-approximation of the PSD cone.}
\begin{tabular}{@{}l l l r r r r@{}} \toprule
Dataset &  $p$ & $k$ & \multicolumn{2}{c@{\hspace{0mm}}}{Minors} &  \multicolumn{2}{c@{\hspace{0mm}}}{Minors + 20 inequalities} \\
\cmidrule(l){4-5} \cmidrule(l){6-7} & &  & Gap $(\%)$ & Time(s) & Gap  $(\%)$ & Time(s) \\\midrule
Pitprops & $13$ & $5$ & $1.51\%$ & $0.02$ &$0.72\%$ &$0.36$ \\
 &  & $10$ & $5.29\%$ & $0.02$ &$1.12\%$ &$0.36$ \\\midrule
Wine & $13$ & $5$ & $2.22\%$ & $0.02$ &$1.59\%$ &$0.38$ \\
 &  & $10$ & $3.81\%$ & $0.02$ &$1.50\%$ &$0.37$ \\\midrule
Miniboone & $50$ & $5$ & $0.00\%$ & $0.11$ &$0.00\%$ &$0.11$ \\
 & & $10$ & $0.00\%$ & $0.12$ &$0.00\%$ &$0.12$ \\
& &  \color{black} $20$ &  \color{black} $0.00\%$ &  \color{black} $0.39$ &  \color{black} $0.00\%$ & \color{black} $0.39$ \\ \midrule
Communities & $101$ & $5$ & $0.07\%$ & $0.67$ &$0.07\%$ &$14.8$ \\
& & $10$ & $0.66\%$ & $0.68$ &$0.66\%$ &$14.4$ \\
& &  \color{black} $20$ &  \color{black} $3.32\%$ &  \color{black} $1.84$ & \color{black} $2.23\%$ & \color{black} $33.5$ \\\midrule
Arrhythmia & $274$ & $5$  & $3.37\%$ & $27.2$ &$1.39\%$ &$203.6$ \\
& & $10$ & $3.01\%$ & $25.6$ &$1.33\%$ &$184.0$ \\
& &  \color{black} $20$ &  \color{black} $8.87\%$ &  \color{black} $21.8$ & \color{black} $4.48\%$ & \color{black} $426.8$ \\\midrule
Micromass & $1300$ & $5$  & $0.04\%$ & $239.4$ &$0.01\%$ &$4,639$ \\
& & $10$ & $0.63\%$ & $232.6$ &$0.32\%$ &$6,392$ \\
& &  \color{black} $20$ &  \color{black} $13.1\%$ &  \color{black} $983.5$ & \color{black} $5.88\%$ & \color{black} $16,350$ \\
\bottomrule
\end{tabular}
\label{tab:comparison_convrelaxations3}
\end{table}

Observe that if we impose constraints on the $2\times 2$ minors only then we obtain a solution {\color{black} certifiably} within {\color{black}$13\%$} of optimality in seconds (resp. minutes) for $p=100$s (resp. $p=1000$s). Moreover, adding $20$ linear inequalities, we obtain a solution within $6\%$ of optimality in minutes (resp. hours) for $p=100$s (resp. $p=1000$s). {\color{black}Moreover, the bound gaps compare favorably to Algorithm \ref{alg:cuttingPlaneMethod} and the method of \cite{berk2017} for instances which these methods could not solve to certifiable optimality. For instance, for the Arrhythmia dataset when $k=20$ we obtain a bound gap of less than $9\%$ in $20$s, while the method of \cite{berk2017} obtains a bound gap of $18.45\%$ in $600$s. This illustrates the value of the proposed relax+round method on datasets which are currently too large to be optimized over exactly.}

To conclude this section, we explore Algorithm \ref{alg:greedymethod}'s ability to scale to even higher dimensional datasets in a high performance setting, by running the method on one Intel Xeon E5--2690 v4 2.6GHz CPU core using 600 GB RAM. Table \ref{tab:comparison_convrelaxations4} reports the methods scalability and performance on the Wilshire $5000$, and \verb|Arcene| UCI datasets. For the \verb|Gisette| dataset, we report on the methods performance when we include the first $3,000$ and $4,000$ rows/columns (as well as all $5,000$ rows/columns). Similarly, for the \verb|Arcene| dataset we report on the method's performance when we include the first $6,000$, $7,000$ or $8,000$ rows/columns. We do not report results for the \verb|Arcene| dataset for $p>8,000$, as computing this requires more memory than was available (i.e. $>600$ GB RAM).
We do not report the method's performance when we impose linear inequalities for the PSD cone, as solving the relaxation without them is already rather time consuming. Moreover, we do not impose the $2 \times 2$ minor constraints to save memory, do not impose $\vert X_{i,j}\vert \leq M_{i,j}z_i$ {\color{black}when $p \geq 4000$} to save even more memory, and report the overall bound gap, as improving upon the randomly rounded solution is challenging in a high-dimensional setting.

\begin{table}[h]
\centering\footnotesize
\caption{Quality of bound gap (rounded solution vs. upper bound) and runtime in seconds.}
\begin{tabular}{@{}l l l r r @{}} \toprule
Dataset &  $p$ & $k$ & \multicolumn{2}{c@{\hspace{0mm}}}{Algorithm \ref{alg:greedymethod} (SOC relax)+Inequalities}\\
\cmidrule(l){4-5}  & &  & Bound gap $(\%)$ & Time(s) \\\midrule
Wilshire $5000$ & $2130$ & $5$ & $0.38\%$ & $1,036$\\
 &  & $10$ & $0.24\%$ & $1,014$\\
  &  &  \color{black} $20$ &  \color{black} $0.36\%$ &  \color{black} $1,059$\\
\midrule
Gisette  & $3000$ & $5$ & $1.67\%$ & $2,249$\\
 &  & $10$ & $35.81\%$ & $2,562$\\
  &  &  \color{black} $20$ &  \color{black} $10.61\%$ &  \color{black} $3,424$\\
\midrule
Gisette  & $4000$ & $5$ & $1.55\%$ & $1,402$\\
 &  & $10$ & $54.4\%$ & $1,203$\\
   &  &  \color{black} $20$ &  \color{black} $11.84\%$ &  \color{black} $1,435$\\
\midrule
Gisette  & $5000$ & $5$ & $1.89\%$ & $2,169$\\
 &  & $10$ & $2.22\%$ & $2,455$\\
  &  &  \color{black} $20$ &  \color{black} $7.16\%$ &  \color{black} $2,190$\\\midrule
Arcene & $6000$ & $5$ & $0.01\%$ & $3,333$\\
 &  & $10$ & $0.06\%$ & $3,616$\\
 &  &  \color{black} $20$ &  \color{black} $0.14\%$ &  \color{black} $3,198$\\\midrule
 Arcene & $7000$ & $5$ & $0.03\%$ & $4,160$\\
 &  & $10$ & $0.05\%$ & $4,594$ \\
   &  &  \color{black} $20$ &  \color{black} $0.25\%$ &  \color{black} $4,730$\\
 \midrule
  Arcene & $8000$ & $5$ & $0.02\%$ & $6,895$\\
 &  & $10$ & $0.17\%$ &  $8,479$\\
   &  &  \color{black} $20$ &  \color{black} $0.21\%$ &  \color{black} $6,335$\\ 
\bottomrule
\end{tabular}
\label{tab:comparison_convrelaxations4}
\end{table}

These results suggest that if we solve the SOC relaxation using a first-order method rather than an interior point method, our approach could successfully generate certifiably near-optimal PCs when $p=10,000$s, particularly if combined with a feature screening technique \citep[see][]{d2008optimal, atamturk2020feature}.

{\color{black}
\subsection{Performance of Exact and Approximate Methods on Synthetic Data}

We now compare the exact and approximate methods against existing state-of-the-art methods in a spiked covariance matrix setting. 
We use the experimental setup laid out in \citet[Section 7.1]{d2008optimal}. We recover the leading principal component of a test matrix\footnote{\color{black}This statement of the test matrix is different to \citet[Section 7.1]{d2008optimal}, who write $\bm{\Sigma}=\bm{U}^\top \bm{U}+\sigma \bm{v}\bm{v}^\top$, rather than $\bm{\Sigma}=\frac{1}{n}\bm{U}^\top \bm{U}+\frac{\sigma}{\Vert \bm{v}\Vert_2^2} \bm{v}\bm{v}^\top$. However, it agrees with their source code.} $\bm{\Sigma} \in S^{p}_+$, where $p=150$,  $\bm{\Sigma}=\frac{1}{n}\bm{U}^\top \bm{U}+\frac{\sigma}{\Vert \bm{v}\Vert_2^2} \bm{v}\bm{v}^\top$, $\bm{U} \in [0, 1]^{150 \times 150}$ is a noisy matrix with i.i.d. standard uniform entries, $\bm{v} \in \mathbb{R}^{150}$ is a vector of signals such that
\begin{align}
    v_i=\begin{cases}
    1, & \text{if} \ i \leq 50,\\
    \frac{1}{i-50}, & \text{if} \ 51 \leq i \leq 100,\\
    0, & \text{otherwise,}
    \end{cases}
\end{align}
and $\sigma=2$ is the signal-to-noise ratio. The methods which we compare are:
\begin{itemize}\itemsep0em
    \item \textbf{Exact}: Algorithm \ref{alg:cuttingPlaneMethod} with Gershgorin inequalities and a time limit of $600$s.
    \item \textbf{Approximate:} Algorithm \ref{alg:greedymethod} with Problem \eqref{spca:sdpplussocp}, the SOC outer-approximation of the PSD cone, no PSD cuts, and the additional SOC inequalities.
    \item \textbf{Greedy:} as proposed by \cite{moghaddam2006spectral} and laid out in \citep[Algorithm 1]{d2008optimal}, start with a solution $\bm{z}$ of cardinality $1$ and iteratively augment this solution vector with the index which gives the maximum variance contribution. Note that \cite{d2008optimal} found this method outperformed the $3$ other methods (approximate greedy, thresholding and sorting) they considered in their work.
    \item \textbf{Truncated Power Method:} as proposed by \citet{yuan2013truncated}, alternate between applying the power method to the solution vector and truncating the vector to ensure that it is $k$-sparse. Note that \cite{berk2017} found that this approach performed better than $5$ other state-of-the-art methods across the real-world datasets studied in the previous section of this paper and often matched the performance of the method of \cite{berk2017}—indeed, it functions as a warm-start for the later method.
    \item \textbf{Sorting:} sort the entries of $\bm{\Sigma}_{i,i}$ by magnitude and set $z_i=1$ for the $k$ largest entries of $\bm{\Sigma}$, as studied in \cite{d2008optimal}. This naive method serves as a benchmark for the value of optimization in the more sophisticated methods considered here.
    \end{itemize}

Figures \ref{fig:sensitivitytok} depicts the ROC curve (true positive rate vs. false positive rate for recovering the support of $\bm{v}$) over $20$ synthetic random instances, as we vary $k$ for each instance. We observe that among all methods, the sorting method is the least accurate, with a substantially larger false detection rate for a given true positive rate than the remaining methods (AUC$=0.7028$). The truncated power method and our exact method\footnote{\color{black}Note that the exact method would dominate the remaining methods if given an unlimited runtime budget. Its poor performance reflects its inability to find the true optimal solution within $600$ seconds.} then offer a substantial improvement over sorting, with respective AUCs of $0.7482$ and $0.7483$. The greedy method then offers a modest improvement over them (AUC$=0.7561$) and the approximate relax+round method is the most accurate (AUC$=0.7593$).

In addition to support recovery, Figure \ref{fig:sensitivitytok2} reports average runtime (left panel) and average optimality gap (right panel) over the same instances. Observe that among all methods, only the exact and the approximate relax+round methods provide optimality gaps, i.e., {\color{black} numerical certificates} of near optimality. On this metric, relax+round supplies average bound gaps of $1\%$ or less on all instances, while the exact method typically supplies bound gaps of $30\%$ or more. This comparison illustrates the tightness of the valid inequalities from Section \ref{ssec:validineq} that we included in the relaxation. Moreover, the relax+round method converges in less than one minute on all instances. All told, the relax+round method is the best performing method overall, although if $k$ is set to be sufficiently close to $0$ or $p$ all methods behave comparably. 
In particular, the relax+round method should be preferred over the exact method, even though the exact method performs better at smaller problem sizes.
\begin{figure}[h]\centering
    \includegraphics[scale=0.6]{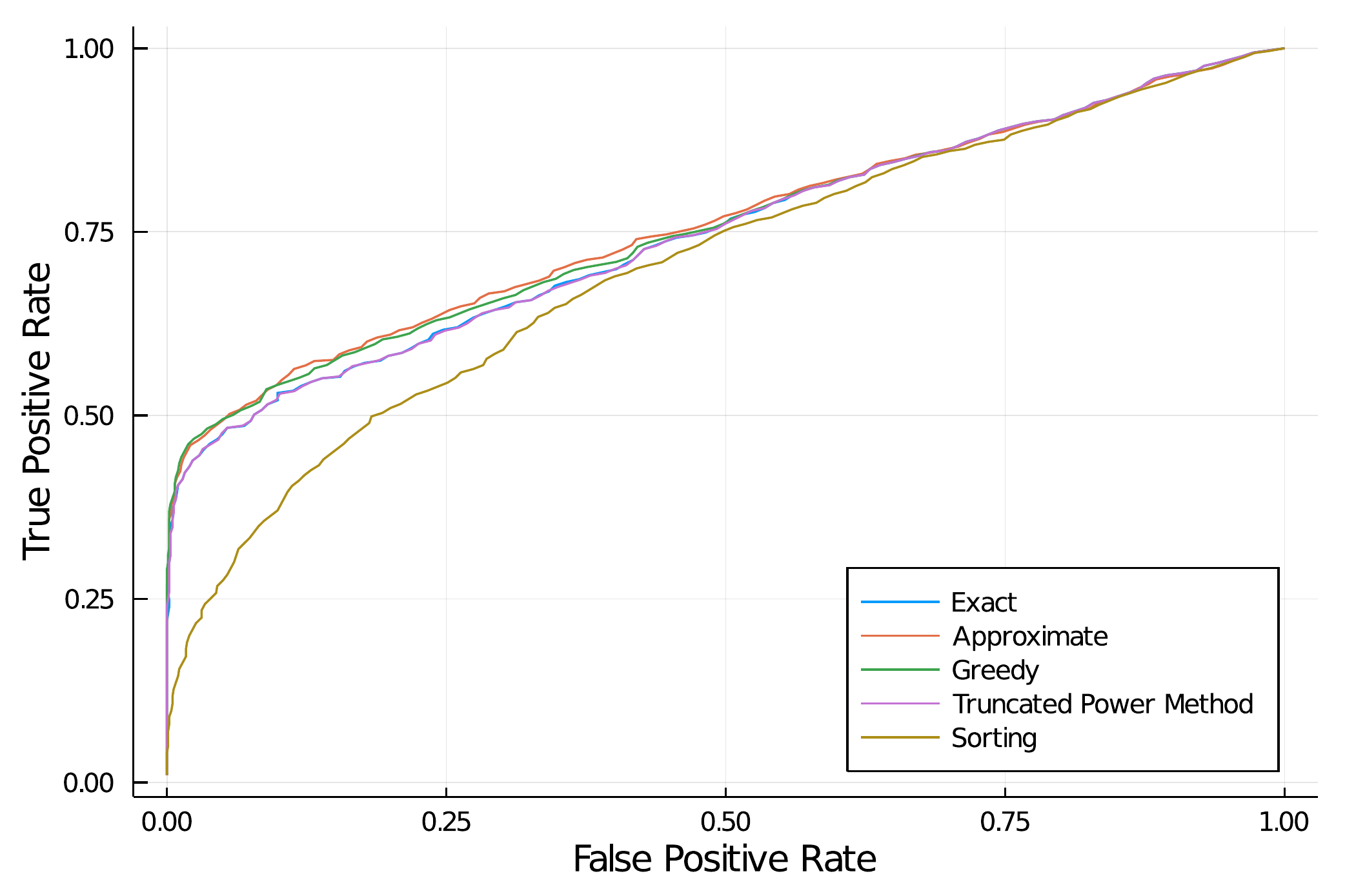}
       \caption{\color{black}ROC curve over $20$ synthetic instances where $p=150$, $k_{\text{true}}=100$ is unspecified.}
   \label{fig:sensitivitytok}
\end{figure}
\begin{figure}[h]
        \begin{subfigure}[t]{.45\linewidth}
            \includegraphics[scale=0.4]{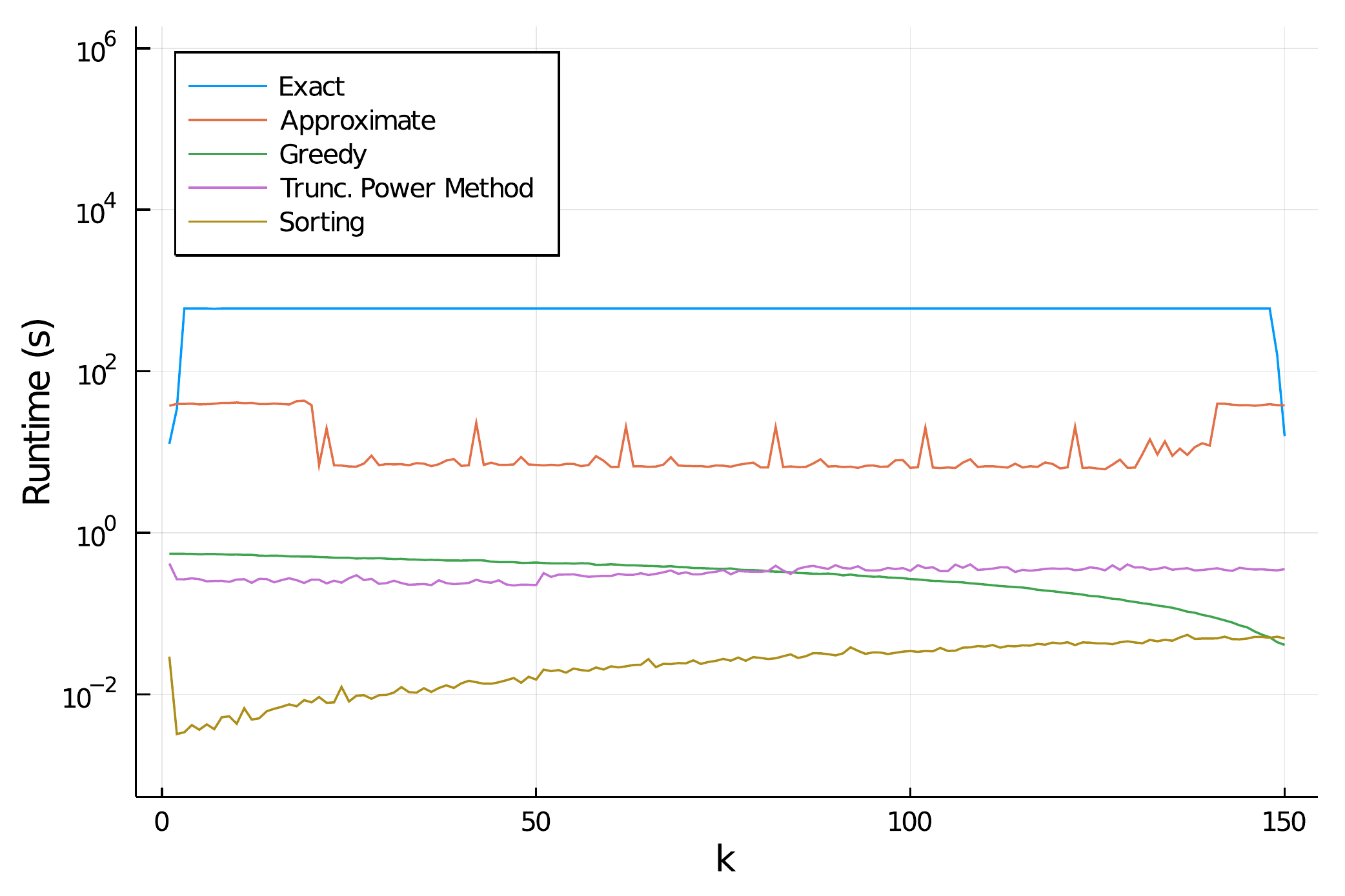}
    \end{subfigure}
    \begin{subfigure}[t]{.45\linewidth}
    \centering
            \includegraphics[scale=0.4]{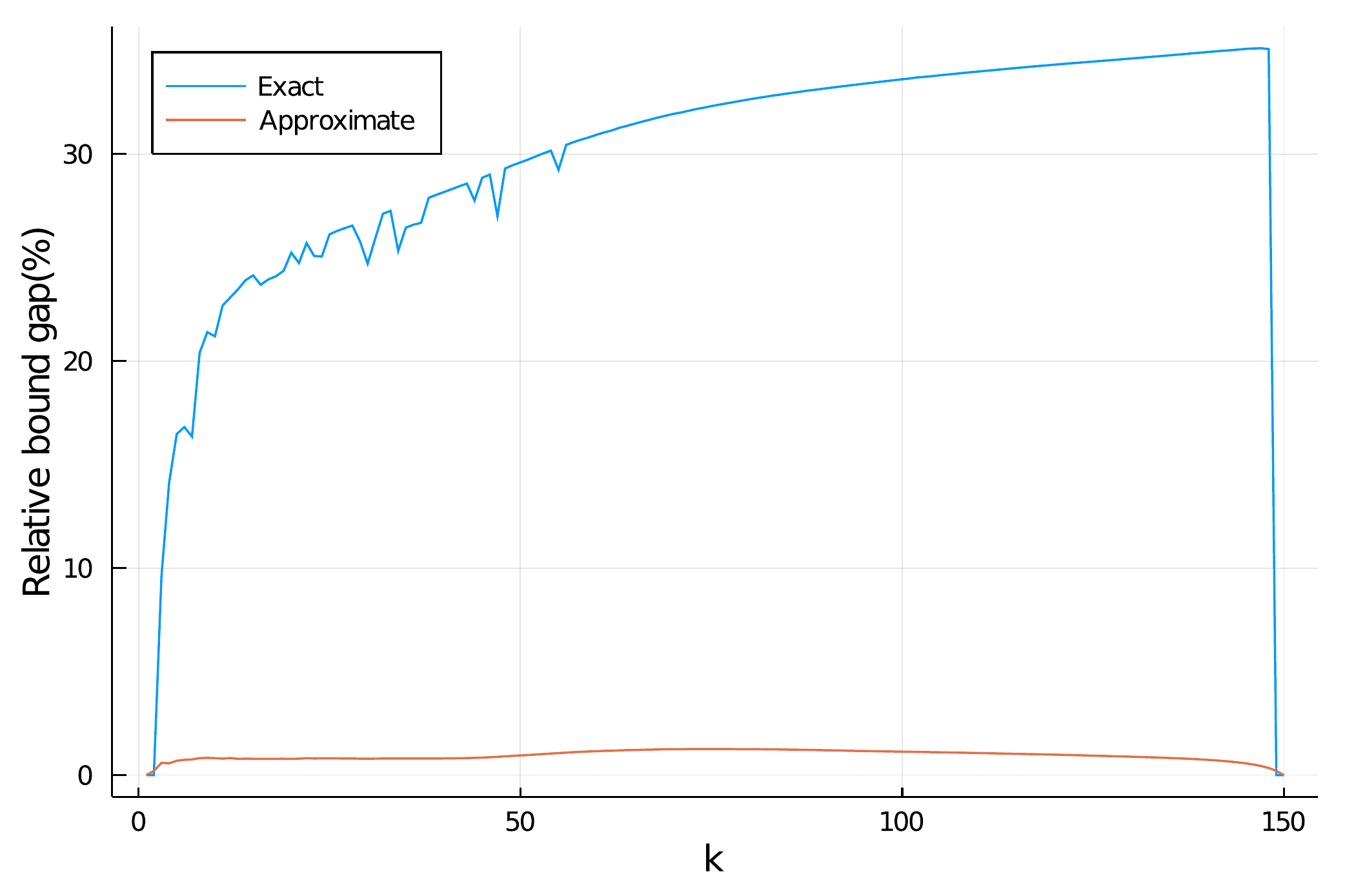}
    \end{subfigure}
   \caption{\color{black}Average time to compute solution, and optimality gap over $20$ synthetic instances where $p=150$, $k_{\text{true}}=100$ is unspecified.}
   \label{fig:sensitivitytok2}
\end{figure}
}
{\color{black}
\subsection{Summary and Guidelines From Experiments}

In summary, our main findings from our numerical experiments are as follows:
\begin{itemize}
\item For small or medium scale problems where $p \leq 100$ or $k \leq 10$, exact methods such as Algorithm \ref{alg:cuttingPlaneMethod} or the method of \cite{berk2017} reliably obtain certifiably optimal or near-optimal solutions in a short amount of time, and should therefore be preferred over other methods. However, for larger-scale sparse PCA problems, exact methods currently do not scale as well as approximate or heuristic methods. \color{black}
\item For larger-scale sparse PCA problems, our proposed combination of solving a second-order cone relaxation and rounding greedily reliably supplies certifiably near-optimal solutions in practice (if not in theory) in a relatively small amount of time. Moreover, it outperforms other state-of-the-art heuristics including the greedy method of \cite{moghaddam2006spectral, d2008optimal} and the Truncated Power Method of \cite{yuan2013truncated}. Accordingly, it should be considered as a reliable and more accurate alternative for problems where $p=1000$s.\color{black}
\item In practice, for even larger-scale problem sizes, we recommend using a combination of these methods: a computationally cheaper method (with $k$ set in the $1000$s) as a feature screening method, to be followed by the approximate relax+round method (with $k$ set in the $100$s) and/or the exact method, if time permits.
\end{itemize}

}

\section{Three Extensions and their Mixed-Integer Conic Formulations}
We conclude by discussing {\color{black}three} extensions of sparse PCA where our methodology applies.

\subsection{Non-Negative Sparse PCA}
One potential extension to this paper would be to develop a certifiably optimal algorithm for non-negative sparse PCA \citep[see][for a discussion]{zass2007nonnegative}, i.e., develop a tractable reformulation of
\begin{align*}
    \max_{\bm{x} \in \mathbb{R}^p} \quad & \langle \bm{x}\bm{x}^\top, \bm{\Sigma} \rangle \ \text{s.t.}\ \bm{x}^\top \bm{x}=1, \bm{x} \geq \bm{0}, \Vert \bm{x}\Vert_0 \leq k.
\end{align*}

Unfortunately, we cannot develop a MISDO reformulation of non-negative sparse PCA \textit{mutatis mutandis} Theorem \ref{thm:misdpreformthm}. Indeed, while we can still set $\bm{X}=\bm{x}\bm{x}^\top$ and relax the rank-one constraint, if we do so then, by the non-negativity of $\bm{x}$, lifting $\bm{x}$ yields:
\begin{equation}\label{misdpprimal_cp}
\begin{aligned}
    \max_{\bm{z} \in \{0, 1\}^p: \bm{e}^\top \bm{z} \leq k} \ \max_{\bm{X} \in \mathcal{C}_n} \quad & \langle \bm{\Sigma}, \bm{X} \rangle\\ \text{s.t.} \quad &  \mathrm{tr}(\bm{X})=1,\ X_{i,j}=0 \ \text{if} \ z_i=0, \ X_{i,j}=0 \ \text{if} \ z_j=0 \ \forall i, j \in [p].
\end{aligned}
\end{equation}
where $\mathcal{C}_n:=\{\bm{X}:\ \exists \ \bm{U} \geq \bm{0}, \bm{X}=\bm{U}^\top \bm{U}\}$ denotes the completely positive cone, which is NP-hard to separate over and cannot currently be optimized over tractably \citep{dong2013separating}. Nonetheless, we can develop relatively tractable mixed-integer conic upper and lower bounds for non-negative sparse PCA. Indeed, we can obtain a fairly tight upper bound by replacing the completely positive cone with the larger doubly non-negative cone $\mathcal{D}_n:=\{\bm{X} \in S^p_+: \bm{X} \geq \bm{0}\}$, which is a high-quality outer-approximation of $\mathcal{C}_n$, indeed exact when $k \leq 4$ \citep{burer2009difference}.

Unfortunately, this relaxation is strictly different in general, since the extreme rays of the doubly non-negative cone are not necessarily rank-one when $k \geq 5$ \citep{burer2009difference}. Nonetheless, to obtain feasible solutions which supply lower bounds, we could inner approximate the completely positive cone with the cone of non-negative scaled diagonally dominant matrices \citep[see][]{ahmadi2019dsos,bostanabad2018inner}.

\subsection{Sparse PCA on Rectangular Matrices}
A second extension would be to extend our methodology to the non-square case:
\begin{align}
    \max_{\bm{x} \in \mathbb{R}^m, \bm{y} \in \mathbb{R}^n} \quad & \bm{x}^\top \bm{A}\bm{y} \ \text{s.t.} \ \Vert \bm{x}\Vert_2=1, \Vert \bm{y}\Vert_2=1, \Vert \bm{x}\Vert_0 \leq k, \Vert \bm{y}\Vert_0 \leq k.
\end{align}

Observe that computing the spectral norm of a matrix $\bm{A}$ is equivalent to:
\begin{align}
\max_{\bm{X} \in \mathbb{R}^{n \times m}}  \quad \langle \bm{A}, \bm{X}\rangle \ \text{s.t.} \ \begin{pmatrix} \bm{U} & \bm{X}\\ \bm{X}^\top & \bm{V}\end{pmatrix} \succeq \bm{0}, \mathrm{tr}(\bm{U})+\mathrm{tr}(\bm{V})=2,
\end{align}
where, in an optimal solution, $\bm{U}$ stands for $\bm{x}\bm{x}^\top$,  $\bm{V}$ stands for $\bm{y}\bm{y}^\top$ and $\bm{X}$ stands for $\bm{x}\bm{y}^\top$—this can be seen by taking the dual of \citep[Equation 2.4]{recht2010guaranteed}.

Therefore, by using the same argument as in the positive semidefinite case, we can rewrite sparse PCA on rectangular matrices as the following MISDO:
\begin{equation}
\begin{aligned}
    \max_{\bm{w} \in \{0, 1\}^m, \bm{z} \in \{0, 1\}^n}\max_{\bm{X} \in \mathbb{R}^{n \times m}}  \quad & \langle \bm{A}, \bm{X}\rangle \\ \text{s.t.} &  \ \begin{pmatrix} \bm{U} & \bm{X}\\ \bm{X}^\top & \bm{V}\end{pmatrix} \succeq \bm{0}, \mathrm{tr}(\bm{U})+\mathrm{tr}(\bm{V})=2,\\ & U_{i,j}=0 \ \text{if} \ w_i=0\ \forall i,j \in [m], \\
    & V_{i,j}=0 \ \text{if} \ z_i=0\ \forall i,j \in [n], \bm{e}^\top \bm{w} \leq k, \bm{e}^\top \bm{z} \leq k.
\end{aligned}
\end{equation}
\subsection{Sparse PCA with Multiple Principal Components}
A third extension where our methodology is applicable is the problem of obtaining multiple principal components simultaneously, rather than deflating $\bm{\Sigma}$ after obtaining each principal component. As there are multiple definitions of this problem, we now discuss the extent to which our framework encompasses each case.

\paragraph{Common Support:} Perhaps the simplest extension of sparse PCA to a multi-component setting arises when all $r$ principal components have common support. By retaining the vector of binary variables $\bm{z}$ and employing the
Ky-Fan theorem \citep[c.f.][Theorem 2.3.8]{wolkowicz2012handbook} to cope with multiple principal components, we obtain the following formulation in much the same manner as previously:
\begin{align}
    \max_{\bm{z} \in \{0, 1\}^p: \bm{e}^\top \bm{z} \leq k}\ \max_{\bm{X} \in S^p_+} \quad & \langle \bm{X}, \bm{\Sigma}\rangle\ \text{s.t.} \ \bm{0} \preceq \bm{X} \preceq \mathbb{I}, \ \mathrm{tr}(\bm{X})=r,\ X_{i,j}=0 \ \text{if} \ z_{i}=0\ \forall i \in [p].
\end{align}
Notably, the logical constraint $X_{i,j}=0$ if $z_i=0$, which formed the basis of our subproblem strategy, still successfully models the sparsity constraint. This suggests that (a) one can derive an equivalent subproblem strategy under common support, and (b) a cutting-plane method for common support should scale equally well as with a single component.

\paragraph{Disjoint Support:}
In a sparse PCA problem with disjoint support \citep{vu2012minimax}
, simultaneously computing the first $r$ principal components is equivalent to solving:
\begin{equation}
\begin{aligned}
    \max_{\substack{\bm{z} \in \{0, 1\}^{p \times r}: \bm{e}^\top \bm{z}_t \leq k\ \forall t \in [r], \\\bm{z}\bm{e} \leq \bm{e}}} \max_{\bm{W} \in \mathbb{R}^{p \times r}} \quad & \langle \bm{W}\bm{W}^\top, \bm{\Sigma}\rangle\\
    & \bm{W}^\top \bm{W}=\mathbb{I}_{r},\ W_{i,j}=0 \ \text{if} \ z_{i,t}=0\ \forall i \in [p], t \in [r],
\end{aligned}
\end{equation}
where $z_{i,t}$ is a binary variable denoting whether feature $i$ is a member of the $t$th principal component. By applying the technique used to derive Theorem \ref{thm:misdpreformthm} \textit{mutatis mutandis}, and invoking the Ky-Fan theorem \citep[c.f.][Theorem 2.3.8]{wolkowicz2012handbook} to cope with the rank-$r$ constraint, we obtain:
\begin{equation}
\begin{aligned}
    \max_{\bm{z} \in \{0, 1\}^p: \bm{e}^\top \bm{z} \leq k} \max_{\bm{X} \in S^p} \quad & \langle \bm{X}, \bm{\Sigma}\rangle\\
    & \bm{0} \preceq \bm{X} \preceq \mathbb{I}, \ \mathrm{tr}(\bm{X})=r,\ X_{i,j}=0 \ \text{if} \ Y_{i,j}=0\ \forall i \in [p],
\end{aligned}
\end{equation}
where $Y_{i,j}=\sum_{t=1}^r z_{i,t}z_{j,t}$ is a binary matrix denoting whether features $i$ and $j$ are members of the same principal component; this problem can be addressed by a cutting-plane method in much the same manner as when $r=1$. 

{\color{black}
\section*{Acknowledgments}
We are grateful to the three anonymous referees and the associate editor for many valuable comments which improved the paper.
}

{\footnotesize
\bibliographystyle{abbrvnat}

}

\begin{appendix}
\section{A Doubly Non-Negative Relaxation and a Goemans-Williamson Rounding Scheme} \label{sec:relax.goemans}

The MISDO formulation \eqref{misdpprimal} we derived in Section \ref{sec:reformulation} features big-$M$ constraints of the form $| X_{i,j} | \leq M_{i,j} z_i$. We did not include the equally valid inequalities $| X_{i,j} | \leq M_{i,j} z_j$, because they are redundant with the fact that $\bm{X}$ is symmetric. Actually,  \eqref{misdpprimal} is equivalent to
 \begin{align}
\max_{\bm{z} \in \{0, 1\}^p: \bm{e}^\top \bm{z} \leq k} \ \max_{\bm{X} \in S^p_+} \quad \langle \bm{\Sigma}, \bm{X} \rangle \text{ s.t. } \quad \mathrm{tr}(\bm{X})=1, \ |X_{i,j}| \leq M_{i,j}z_i z_j, \ \forall i, j \in [p].
\end{align}
The formulation above features products of binary variables $z_i z_j$.
Therefore, unlike several other problems involving cardinality constraints such as compressed sensing, relaxations of sparse PCA benefit from invoking an optimization hierarchy \citep[see][Section 2.4.1, for a counterexample specific to compressed sensing]{d2003relaxations}. In particular, let us model the outer product $\bm{z}\bm{z}^\top$ by introducing a matrix $\bm{Z}$ and imposing the semidefinite constraint $\bm{Z} \succeq \bm{z}\bm{z}^\top$. We tighten the formulation by requiring that $Z_{i,i}=z_i$ and imposing the linear inequalities $\max(z_i+z_j-1,0) \leq Z_{i,j} \leq \min(z_i, z_j)$. Hence, we obtain:
\begin{align}\label{prob:dnnrelax2}
\max_{\substack{\bm{z} \in [0, 1]^p:\bm{e}^\top \bm{z} \leq k, \\\bm{Z} \in \mathbb{R}^{p \times p}_+}} \: \max_{\bm{X} \succeq \bm{0}} \: \langle \bm{\Sigma}, \bm{X} \rangle \text{ s.t. }
& \mathrm{tr}(\bm{X})=1,\ \vert X_{i,j}\vert \leq M_{i,j}Z_{i,j}, \ \langle \bm{E}, \bm{Z} \rangle \leq k^2, \: Z_{i,i}=z_i, \\
& \max(z_i+z_j-1,0) \leq Z_{i,j} \leq \min(z_i, z_j), \nonumber\ \begin{pmatrix} 1 & \bm{z}^\top \\ \bm{z} & \bm{Z} \end{pmatrix} \succeq \bm{0}. \nonumber
\end{align}

Problem \eqref{prob:dnnrelax2} is a doubly non-negative relaxation, as we have intersected the Shor and RLT relaxations. This is noteworthy, because doubly non-negative relaxations dominate most other popular relaxations with $O(p^2)$ variables \citep[Theorem 1]{bao2011semidefinite}.

Relaxation \eqref{prob:dnnrelax2} is amenable to a \textit{Goemans-Williamson} rounding scheme \citep{goemans1995improved}. Namely, let $(\bm{z}^\star, \bm{Z}^\star)$ denote optimal choices of $(\bm{z}, \bm{Z})$ in Problem \eqref{prob:dnnrelax2}, $\hat{\bm{z}}$ be normally distributed random vector such that $\hat{\bm{z}} \sim \mathcal{N}(\bm{z}^\star, \bm{Z}^\star-\bm{z}^\star \bm{z}^{\star\top})$, and $\bar{\bm{z}}$ be a rounding of the vector such that $\bar{z_i}=1$ for the $k$ largest entries of $\hat{z}_i$; this is, up to feasibility on $\hat{\bm{z}}$, equivalent to the hyperplane rounding scheme of \citet{goemans1995improved}  \citep[see][for a proof]{bertsimas1998semidefinite}. We formalize this procedure in Algorithm \ref{alg:gwmethod}. As Algorithm \ref{alg:gwmethod} returns one of multiple possible $\bar{\bm{z}}$'s, a computationally useful strategy is to run the random rounding component several times and return the best solution.

\begin{algorithm*}[h]
\caption{A Goemans-Williamson rounding method for Problem \eqref{OriginalSPCA}}
\label{alg:gwmethod}
\begin{algorithmic}\normalsize
\REQUIRE Covariance matrix $\bm{\Sigma}$, sparsity parameter $k$
\STATE Compute $\bm{z}^\star, \bm{Z}^\star$ solution of \eqref{prob:dnnrelax2}
\STATE Compute $\hat{\bm{z}} \sim \mathcal{N}(\bm{z}^\star, \bm{Z}^\star-\bm{z}^\star \bm{z}^{\star\top})$
\STATE Construct $\bar{\bm{z}} \in \{0, 1\}^p: \bm{e}^\top \bar{\bm{z}}=k$ such that $\bar{z}_i\geq \bar{z}_j$ if $\hat{z}_i \geq \hat{z}_j$.
\STATE Compute $\bm{X}$ solution of \vspace{-2mm}
\begin{align*}
    \max_{\bm{X} \in S^p_+} \ \langle \bm{\Sigma}, \bm{X} \rangle \ \text{s.t.} \ \mathrm{tr}(\bm{X})=1, X_{i,j}=0 \ \text{if} \ \bar{z}_i \bar{z}_j=0 \ \forall i,j \in [p].
\end{align*} \vspace{-5mm}
\RETURN $\bm{z}, \bm{X}$.
\end{algorithmic}
\end{algorithm*}

A very interesting question is whether it is possible to produce a constant factor guarantee on the quality of Algorithm \ref{alg:gwmethod}'s rounding, as \citet{goemans1995improved} successfully did for binary quadratic optimization. Unfortunately, despite our best effort, this does not appear to be possible as the quality of the rounding depends on the value of the optimal dual variables, which are hard to control in this setting. This should not be too surprising for two distinct reasons. Namely, (a) sparse regression, which reduces to sparse PCA \citep[see][Section 6.1]{d2008optimal} is strongly NP-hard \citep{chen2019approximation}, and (b) sparse PCA is hard to approximate within a constant factor under the Small Set Expansion (SSE) hypothesis \citep{chan2016approximability}, meaning that producing a constant factor guarantee would contradict the SSE hypothesis of \citet{raghavendra2010graph}.

We close this appendix by noting that a similar in spirit (although different in both derivation and implementation) combination of taking a semidefinite relaxation of $\bm{z} \in \{0, 1\}^p$ and rounding \textit{\`a la} Goemans-Williamson has been proposed for sparse regression problems \citep{dong2015regularization}.

\end{appendix}

\end{document}